\documentclass[final]{siamltex}
\usepackage{amsmath}
\usepackage{mathtools}
\usepackage{amssymb}
\usepackage{amsfonts}
\usepackage{amsxtra}
\usepackage{amstext}
\usepackage{amsbsy}
\usepackage{amscd}
\usepackage{graphicx}
\usepackage{float}
\usepackage{cite}
\usepackage{xcolor}
\usepackage{srcltx} 
\usepackage{marginnote,slashbox}
\usepackage{rotating}
\definecolor{darkblue}{rgb}{0.0,0.0,0.6}
\definecolor{darkgreen}{rgb}{0.0,0.6,0.0}
\usepackage[pdftex,colorlinks=true,urlcolor=darkblue,citecolor=darkblue,linkcolor=darkblue]{hyperref}
\usepackage[utf8]{inputenc}      
\usepackage{booktabs}            
\usepackage{url}                 
\usepackage[T1]{fontenc}         
\usepackage{algorithmic}         

\usepackage{calc}
\usepackage[notcite,notref]{showkeys}

\numberwithin{table}{section}    
\numberwithin{figure}{section}   
\numberwithin{equation}{section} 

\usepackage{my_latex_commands}

\newtheorem{assumption}[theorem]{Assumption}
\newtheorem{remark}[theorem]{Remark}
\newenvironment{proofof}[1]{\emph{Proof of #1.}}{\endproof}

\newcommand{\maxeps}{\maxlim{_\varepsilon}}
\newcommand{\maxlim}[1]{\max\nolimits{#1}}

\newcommand{\dl}{\delta \ell}
\newcommand{\eps}{\varepsilon}
\newcommand{\df}{\delta \varphi}

\newcommand{\dd}{\delta q}
\newcommand{\yy}{q}

\renewcommand{\ll}{\bar \ell}

\newcommand{\lo}{L^2(\O)}
\newcommand{\ho}{H^1(\O)}
\newcommand{\hoo}{H^1(\O)^\ast}

\newcommand{\lii}{L^\infty((0,T)  \times \O)}
\newcommand{\li}{L^\infty( \O)}

\newcommand{\dense}{\overset{d}{\embed}}

\newcommand{\llU}{L^2(0,T;\hoo)}
\newcommand{\llu}{L^2(0,T;H^{1}(\Omega))}
\newcommand{\llz}{L^2(0,T; L^2(\Omega))}

\newcommand{\hhyn}{H^1_0(0,T;\lo)}
\newcommand{\hs}{H^1(0,T;\lo)}

\newcommand{\F}{(f \circ \HH)}
\begin{document}

\title{Optimal control of a viscous two-field damage model with fatigue
}
\date{\today}
\author{Livia\ Betz\footnotemark[1]}
\renewcommand{\thefootnote}{\fnsymbol{footnote}}
\footnotetext[1]{Faculty of Mathematics, University of W\"urzburg,  Germany.}
\renewcommand{\thefootnote}{\arabic{footnote}}

\maketitle
 
 \begin{abstract}
 Motivated by  fatigue damage models,  this paper addresses  optimal control problems governed by a non-smooth system featuring two non-differentiable mappings. This consists of a coupling between a doubly non-smooth history-dependent evolution and an elliptic PDE. After proving the directional differentiability of the associated solution mapping, an optimality system which is  stronger than the one obtained by classical smoothening procedures is derived. If one of the non-differentiable mappings becomes smooth, the optimality conditions are of strong stationary type, i.e., equivalent to the primal necessary optimality condition.\end{abstract}



\section{Introduction}\label{sec:i} 
Fatigue is considered to be the main cause of mechanical failure \cite{rl, stephens}.  It describes the weakening of a material due to repeated applied loads (fluctuating stresses, strains, forces, environmental factors, temperature, etc.), which individually would be too small to cause its malfunction \cite{alessi, stephens}.
Whether in association with environmental damage (corrosion fatigue) or elevated temperatures (creep fatigue), fatigue failure is often an   unexpected phenomenon. Unfortunately, in real situations, it is very difficult to identify the fatigue degradation state of a material, which sometimes might  result in devastating events.  Therefore, it is extremely important to find methods which allows us to describe and  control the behaviour of  materials exposed to fatigue. While there are very few papers  \cite{alessi} (damage in elastic materials)  and \cite{a1} (cohesive fracture), concerned with a rigurouos mathematical examination of models describing fatigue damage, the literature regarding the optimal control of fatigue models is practically nonexistent. All the existing results  which include the terminology ``optimal control'' in the context of fatigue damage do not address theoretical aspects nor involve  mathematical tools such as optimal control theory in Banach spaces as in the present  work, but focus on design of controllers and simulations instead, see e.g.\ \cite{oc1,oc2} and the references therein. 

In this paper we investigate the optimal control of the following viscous two-filed gradient damage problem with fatigue:
\begin{equation}\label{eq:q}
\left.
  \begin{gathered}
   \varphi(t) \in \argmin_{\varphi \in H^1(\Omega)} \EE(t,  \varphi, q(t)),
    \\ -\partial_q \EE(t,\varphi(t),q(t)) \in \partial_{\dot{q}} \RR_\epsilon (\HH(q)(t),\dot{q}(t)) \text{ in }\lo , \quad q(0) = 0,
  \end{gathered}\ \right\}
  \end{equation}a.e.\ in $(0,T)$. 
To be more precise, we prove an optimality system that is far stronger than the one obtained by classical smoothening techniques. 
  
  The main novelty concerning  \eqref{eq:q} arises from the \textit{highly non-smooth  structure}, which is  due to the non-differentiability of the dissipation $\RR_\epsilon$ in the evolution inclusion, in combination with \textit{an additional non-smooth}  fatigue degradation mapping which shall be introduced below. This excludes the application of standard adjoint techniques for the derivation of first-order necessary conditions in form of optimality  systems. 
Not only does the evolution in \eqref{eq:q} have a highly non-smooth character, but, as we will next see, it is also \textit{history-dependent}. The fact that the differential inclusion is coupled with a minimization problem (which can be reduced to an elliptic PDE) gives rise to additional challenges \cite{st_coup}.

The problem describes the evolution of damage under the influence of a time-dependent load $\ell:[0,T] \to H^1(\Omega)^\ast$ (control) acting on a body occupying the bounded Lipschitz domain $\O \subset \R^N$, $N \in \{2,3\}$. The induced 'local' and 'nonlocal' damage are expressed in terms of the functions  $q : [0,T] \to L^2(\Omega)$ and $\varphi : [0,T] \to H^{1}(\Omega)$, respectively (states). 

   In \eqref{eq:q}, the stored energy $\EE:[0,T] \times H^{1}(\Omega) \times L^2(\Omega)\rightarrow \mathbb{R}$ is given by
\begin{equation}\label{eq:e}\begin{aligned}
\EE(t, \varphi,q)&:=\frac{\alpha}{2}\|\nabla \varphi\|_{L^2(\O)}^2+\frac{\beta}{2}\| \varphi - q\|_{L^2(\O)}^2-\dual{\ell(t)}{\varphi}_{H^{1}(\Omega)},
\end{aligned}\end{equation}where $\alpha>0$ is the gradient regularization and $ \beta >0$ denotes the \textit{penalization parameter}. 
Thus, the two damage variables are connected through the penalty term $\beta$ in the stored energy, so that our model becomes  a penalized version of the viscous fatigue damage model addressed in \cite{alessi} (two-dimensional case);  note that, for simplicity reasons, we do not take a displacement variable into account. 
The type of penalization used in \eqref{eq:e} has already been proven to be successful  in the context of classical damage models (without fatigue). Firstly, it approximates the classical  single-field damage model, in the sense that, when $\beta \to \infty$, the penalized damage model coincides with the model addressed in  \cite{fn96, knees},  cf.\ \cite{paper2}. Secondly, the  penalization we use is  frequently employed in computational mechanics due to the numerical benefits offered by the additional  damage variable (see e.g.\ \cite{hackl} and the references therein). For more details, we also refer to \cite[Sec.\ 2.1-2.2]{paper1}.

The differential inclusion appearing in \eqref{eq:q} describes the evolution of the damage variable $q$
under \textit{fatigue} effects. Therein, $\HH$ is a so-called \textit{history operator}  that models how the  damage experienced by the material affects its fatigue level. 
Thus, as opposed to other  well-known damage models, cf.\ e.g.\  \cite{fn96, FK06,  knees},  the dissipation $\RR_\epsilon$ in \eqref{eq:q} is affected  by the \textit{history of the evolution}, $\HH(q)$. The parameter $\epsilon>0$ stands for the viscosity parameter, while the symbol $\partial_{\dot q}$ denotes the convex subdifferential of the functional $\RR_\epsilon$ in its second argument. Thus, the \textit{non-smooth differential inclusion} is to be understood as follows:
  $$(-\partial_q \EE(t,\varphi(t),q(t)), \eta-\dot{q}(t))_{\lo} \leq \RR_\epsilon (\HH(q)(t),\eta)- \RR_\epsilon (\HH(q)(t),\dot{q}(t)) \quad \forall\,\eta \in \lo.$$
  The viscous dissipation  $\RR_\epsilon:L^2(\O) \times  L^2(\O) \rightarrow \R$ is defined as \begin{equation}\label{def:r}
\RR_\epsilon(\omega,\eta):=\left\{ \begin{aligned}\int_\Omega f(\omega) \,  \eta \;dx +\frac{\epsilon}{2}\|\eta\|^2_{\lo}, &\quad \text{if }\eta \geq 0 \text{ a.e. in }\Omega,
\\\infty  &\quad \text{otherwise,}\end{aligned} \right.
\end{equation}and features \textit{a second non-smooth component}, namely the fatigue degradation mapping $f$.  This describes in which measure the fatigue affects the fracture toughness  of the material. This mapping is non-increasing in applications, since the higher the cumulated damage $\HH(q)$, the lower the fracture toughness $f(\HH(q))$. Whereas usually the toughness of the material  is described by a fixed  (nonnegative) constant \cite{fn96, FK06}, in the present model it changes at each point in time and space, depending on $\HH(q)$. To be more precise, the value of the fracture toughness of the body at $(t,x)$ is given by  $f(\HH(q))(t,x)$, cf.\ \eqref{def:r}. Hence, the model \eqref{eq:q} takes into account the following crucial aspect: the occurrence of damage is favoured in regions where fatigue accumulates.

 We underline that the dissipation $\RR_\epsilon$ accounts for the non-smooth nature of the evolution in the first place: even if $f$ is replaced by a (nonnegative) constant, the evolution in \eqref{eq:q} still describes a non-smooth process. The optimal control thereof is far away from being standard and  has been recently addressed in  \cite[Sec.\,4]{st_coup}, where strong stationarity for the damage model \eqref{eq:q} without fatigue is proven. By contrast, in applications which take fatigue into consideration, $f:\R^+ \to \R^+$ is constant until its  kink point is achieved, after which it monotonically decreases \cite[Sec.\ 2.6.2]{alessi1}. Thus, it is  the fatigue degradation mapping $f$ which accounts for the \textit{highly non-smooth character} of our problem.

Deriving necessary optimality conditions is a challenging issue even in finite dimensions, where a special attention is given to MPCCs  (mathematical programs with complementarity constraints).
In \cite{ScheelScholtes2000} a detailed overview of various optimality conditions of different strength was introduced, see also \cite{HintermuellerKopacka2008:1} for the infinite-dimensional case. The  most rigorous stationarity concept is strong stationarity.
Roughly speaking, the strong stationarity conditions involve an optimality system, which is equivalent to the purely primal conditions saying that the directional derivative of the reduced objective in feasible directions is nonnegative (which is referred to as B stationarity).

While there are plenty of contributions in the field of optimal control of 
smooth problems, see e.g.\ \cite{troe} and the references therein, {fewer} papers are dealing with 
non-smooth problems. Most of these papers resort to regularization or relaxation techniques to smoothen the problem, see e.g.\ \cite{ Barbu:1981:NCD,  Friedman1987, He1987} and the references therein. The optimality systems derived in this way are of intermediate strength and are not expected to be of strong stationary type, since one always loses information when passing to the limit in the regularization scheme. Thus, proving strong stationarity for optimal control of  non-smooth problems requires direct approaches, which employ the limited differentiability properties of the control-to-state map. In this context, there are even less contributions. We  refer to the pioneering work \cite{mp76} (strong stationarity  for  optimal control of elliptic VIs of obstacle type), which was followed by other papers addressing strong stationarity of various types of VIs \cite{mp84, wachsm_2014, DelosReyes-Meyer, cc, e_qvi, brok_ch}. Regarding strong stationarity for  optimal control of non-smooth PDEs, the literature is rather scarce and the only  papers known to the author addressing this issue so far are  \cite{paper, cr_meyer, mcrf, st_coup, quasi_nonsmooth}.

Let us point out the main contributions of the present work. This paper aims at deriving optimality conditions which - regarding their strength - lie between the conditions  derived by classical regularization techniques and the strong stationary ones.
Starting from an optimality system obtained via smoothening, we resort to direct methods  from   previous works \cite{st_coup, paper}, in order to improve our initial optimality conditions as far as we can. Note that this is a novel way of obtaining optimality conditions. We emphasize that, in contrast to \cite{st_coup, paper}, our state system features \textit{two} non-differentiable mappings instead of one, so that the methods from the aforementioned works are of limited applicability: Strong stationary conditions are not expected in our complex doubly non-smooth setting. If the fatigue degradation mapping is smooth, strong stationarity conditions are indeed available.
We underline that, to the best of our knowledge, optimal control problems featuring two non-differentiable functions have not been tackled so far in the literature, not even in the context of classical smoothening methods.

The paper is structured as follows. After an introduction of  the notation, section \ref{sec:c} focuses on the analysis of our fatigue damage model \eqref{eq:q}. Here we address the existence and uniqueness of solutions, by proving that \eqref{eq:q} is in fact equivalent to a PDE system. This consists of an elliptic PDE and a  highly non-smooth differential ODE. The latter one is of particular interest. It features two non-differentiable functions, namely $\max$ and the fatigue degradation function $f$; the latter appears in the argument of the initial non-smoothness, cf\,\eqref{eq:ode}. The properties of the  control-to-state operator associated to \eqref {eq:q} are investigated. In particular, we are concerned with the \textit{directional differentiability} of the solution mapping of the non-smooth state system. To the best of our knowledge, the sensitivity analysis of non-smooth differential equations containing two non-differentiable functions has never been examined in the literature. 

In section \ref{sec:o} we present the optimal control problem and investigate the existence of optimal minimizers. Then, in subsection \ref{sec:reg} we derive our first optimality conditions, by resorting to a classical smoothening method. These conditions are of intermediate strength. If the non-smoothness is inactive, they coincide with the classical KKT system. However, our first optimality system does not contain any information in those points $(t,x)$ where the non-differentiable mappings $\max$ and $f$ attain their kink points. This is namely the focus of section \ref{sec:t}, where the main result is proven in    Theorem \ref{thm:ss_qsep}. Here, the initial optimality system is improved  by employing  the  "surjectivity" trick from \cite{st_coup, paper}. The new and final optimality conditions \eqref{eq:strongstat_q} are comparatively strong (but not strong stationary). They contain information in terms of sign conditions on sets where the non-smoothness is active; these are not expected to be obtained if one just smoothens the problem, cf.\,e.g. \cite[Remark 3.9]{mcrf}.  Moreover, if the fatigue degradation function $f$ is smooth, then  \eqref{eq:strongstat_q} is of strong stationary type (Corollary \ref{cor}). For completeness, the expected (not proven) strong stationarity system associated to the doubly non-smooth state system  is presented in Section \ref{sec:comp}. Here we include a  thorough explanation as to why the methods from \cite{paper, st_coup} fail (Remark \ref{exp}). Finally, we include in Appendix \ref{sec:a}  the proof of Lemma \ref{lem:loc}, for convenience of the reader.

\subsection*{Notation}
Throughout the paper, $T > 0$ is a fixed final time. If $X$ and $Y$ are linear normed spaces, then the space of linear and bounded operators from 
$X$ to $Y$ is denoted by $\LL(X,Y)$, and $X \overset{d}{\embed} Y$ means that $X$ is densely embedded in $Y$. The dual space of  $X$ will be denoted by $X^*$. For the dual pairing between $X$ and $X^*$
we write $\dual{.}{.}_X$.  The closed ball in $X$ around $x \in X$ with radius $\alpha >0$ is denoted by $B_X(x,\alpha)$. If $X$ is a Hilbert space, we write $(\cdot,\cdot)_X$ for the associated scalar product.
The  following abbreviations will be used throughout the paper:
\begin{equation*}
\begin{aligned}
H^{1}_0(0,T;X)&:=\{z\in H^{1}(0,T;X):z(0)=0\},
\\H^{1}_T(0,T;X)&:=\{z\in H^{1}(0,T;X):z(T)=0\},
\end{aligned}
\end{equation*}where $X$ is a Banach space. The adjoint operator of a linear and continuous mapping $A $ is denoted by $A^\star.$
By  $\raisebox{3pt}{$\chi$}_{M}$ we denote the characteristic function associated to the set $M$. Derivatives w.r.t.\ time (weak derivatives of vector-valued functions) are frequently denoted by a dot.
The symbol $\partial$ stands for the convex subdifferential, see e.g.\ \cite{rockaf}.
With a little abuse of notation, the Nemystkii-operators associated with the mappings considered in this paper will be denoted by the same symbol, even when considered with different domains and ranges. The mapping $ \max \{ \cdot , 0\}$ is abbreviated by $\max(\cdot)$.  With a little abuse of notation, we use in the following  the Laplace symbol for the operator $\Delta: \ho \to \hoo$ defined by 
\begin{equation*}
 \dual{\Delta \eta}{\psi}_{\ho} := - \int_\Omega \nabla \eta  \nabla \psi \,dx \quad \forall\,
 \psi \in \ho.
\end{equation*}

\section{Properties of the control-to-state map}\label{sec:c}
This section is concerned with the investigation of the solvability and  differentiability properties of the state system \eqref{eq:q}. 

\begin{assumption}\label{assu:stand}For the mappings associated with fatigue  in \eqref{eq:q} we require the following:
 \begin{enumerate}
  \item\label{it:stand1}  The \textit{history operator} $\HH:\llz \to \llz$ satisfies 
$$\|\HH (q_1)(t)-\HH (q_2)(t)\|_{\lo} \leq L_\HH\, \int_0^t \|q_1(s)-q_2(s)\|_{\lo} \,ds \quad \ae \text{in } (0,T),$$for all  $q_1,q_2 \in \llz$, where $L_\HH>0$ is a positive constant.   Moreover, $ \HH:  \llz \to \llz$ is supposed to be G\^ateaux-differentiable with continuous derivative on $H^1(0,T;\lo)$.
\item\label{it:stand2}
The non-linear function  $f: \R \to \R$ is assumed to be Lipschitz-continuous with Lipschitz-constant $L_f>0$ and directionally differentiable.  
 \end{enumerate}
\end{assumption}
{\begin{remark}\label{rem:volt}
Assumption \ref{assu:stand}.\ref{it:stand1} is satisfied by the Volterra operator $\HH:\llz \to \llz$, defined as 
\begin{equation*}
 [0,T] \ni t \mapsto \HH(q)(t):=\int_0^t A(t-s)q(s) \,ds +q_0 \in \lo,\end{equation*}
where $A\in C([0,T]; \LL(\lo,\lo))$ and $q_0\in \lo.$
This type of operator is often employed in the study of history-depedent evolutionary variational inequalities, see e.g.\cite[Ch.\,4.4]{sofonea}.

Concerning Assumption \ref{assu:stand}.\ref{it:stand2}, we remark that non-differentiable fatigue degradation functions are very common in applications, since such mappings often display at least one  kink point, see \cite[Sec.\ 2.6.2]{alessi1}. {This basically means that once the cumulated fatigue $\HH(q)$ achieves a certain value, say $n_f$, the body suddenly starts to become weaker in terms of its fracture toughness (so that $n_f$ is a kink point of $f$). This abrupt weakening of the material is described by the monotonically decreasing mapping $f$ on the interval $[n_f,\infty)$, see \cite[Sec.\ 2.6.2]{alessi1}.}
\end{remark}}

Assumption \ref{assu:stand} is supposed to hold throughout the paper, without mentioning it every time. 


It is not difficult to check that the Nemytskii operator $f:L^2(\O) \to \lo$ is Lipschitz continuous with constant $L_f$. In view of Assumption \ref{assu:stand}.\ref{it:stand1}, we thus have  
\begin{equation}\label{eq:f_h}
\|(f \circ \HH)(q_1)(t)- (f \circ \HH)(q_2)(t)\|_{\lo}
\leq L_f    \, L_{\HH} \int_0^t \|q_1(s)-q_2(s) \|_{\lo}  \;ds 
\end{equation}\text{\ae }\text{in} $(0,T)$, for all  $q_1,q_2 \in \llz$. 
\begin{proposition}[Control-to-state map]\label{lem:ode}
For every $\ell \in \llU$, the fatigue damage problem \eqref{eq:q} admits a unique solution $(q,\varphi) \in   H^1_0(0,T;\lo) \times \llu $, which is characterized by  the following PDE system
\begin{subequations}\label{eq:syst_diff}
\begin{gather}
\dot q(t)   =   \frac{1}{\epsilon} \max \big(-\beta(q(t)-\varphi(t))-(f \circ \HH)(q)(t)\big)\ \text{ in }\lo, \quad q(0) = 0,    \label{eq:ode} 
\\- \alpha \Delta \varphi(t) + \beta \,\varphi(t)   
= \beta q(t) +\ell(t) \quad  \text{in }H^{1}(\Omega)^\ast \label{eq:el} 
\end{gather}
\end{subequations}a.e.\ in $(0,T)$. 
    \end{proposition}

      \begin{proof}
      Let $t \in [0,T]$ and $\hat q: [0,T] \to \lo$ be arbitrary, but fixed. Since $\EE(t,\cdot,\hat q(t))$ is strictly convex, continuous and radially unbounded (see \eqref{eq:e}), the minimization problem $\min_{\varphi \in \ho} \EE(t,\cdot,\hat q(t))$ admits a unique solution $\hat \varphi(t)$ characterized by 
      $\partial_\varphi \EE(t, \hat \varphi(t), \hat q(t))=0 \ \text{in }\hoo$. In view of \eqref{eq:e}, this means that
\begin{equation}\label{eq:ell_eq}        
\hat    \varphi(t) \in \argmin_{\varphi \in H^1(\Omega)} \EE(t,  \varphi,\hat  q(t)) \Longleftrightarrow \hat \varphi(t)=\phi(\hat q(t),\ell(t)),
   \end{equation}where $\phi:\lo \times \hoo \ni (\widetilde q,\widetilde \ell) \mapsto \widetilde  \varphi \in \ho$ is the solution operator of 
          \begin{equation}\label{eq:ell}  
          - \alpha \Delta  \widetilde \varphi+ \beta \,\widetilde \varphi
= \beta \widetilde q +\widetilde \ell \quad  \text{in }H^{1}(\Omega)^\ast .
\end{equation} 
  With the map $\phi$ at hand, the evolution in \eqref{eq:q} reads 
          \begin{equation}\label{eq:ev0}  
     -\partial_q \EE(t,\phi(q(t),\ell(t)),q(t)) \in \partial_{\dot q} \RR_\epsilon (\HH(q)(t),\dot{q}(t))  \ \text{ a.e.\ in }(0,T).
            \end{equation}
            In the light of \eqref{eq:e}, \eqref{def:r}, and sum rule for convex subdifferentials, \eqref{eq:ev0} is equivalent to 
    \begin{equation}\label{eq:q_red}  
    \RR(\HH(q)(t),v)-\RR(\HH(q)(t),\dot{q}(t))+\epsilon\,(\dot{q}(t),v-\dot{q}(t))_{\lo}  \geq \beta \big(\phi(q(t),\ell(t))-q(t),v-\dot{q}(t)\big)_{\lo}       \end{equation}for all $v \in L^2(\Omega),$ a.e.\ in $(0,T)$,
where \begin{equation}\label{eq:R}\RR: L^2(\Omega)  \times L^2(\Omega) \rightarrow \R, \quad
\RR(\omega, \eta):=\begin{cases}\int_\Omega f(\omega) \eta \;dx, &\text{if }\eta \geq 0 \text{ a.e. in }\Omega,
\\\infty  &\text{otherwise.}\end{cases}
\end{equation}
Now we use the result in \cite[Lemma 3.3]{st_coup} for each time point $t$ and we see that \eqref{eq:q_red} is in fact equivalent with
   \begin{equation}\label{ode}
   \dot q(t)=\frac{1}{\epsilon} (\mathbb{I}-P_{\partial_{\dot q} \RR (\HH(q)(t),0)}) \big(g(q(t),\ell(t))\big) \ \text{ a.e.\ in }(0,T),\end{equation}
 where we abbreviate for convenience 
 \begin{equation}\label{g}
 g(q(t),\ell(t)):=\beta(\phi(q(t),\ell(t))-q(t)).
 \end{equation}
 In \eqref{ode}, $P_{\partial_{\dot q} \RR (\HH(q)(t),0)}:\lo \to \lo $ stands for the (metric) projection onto the set $\partial_{\dot q} \RR (\HH(q)(t),0)$, i.e., $P_{\partial_{\dot q} \RR (\HH(q)(t),0)} \eta$ is the unique solution of \begin{equation*}\label{eq:min} \min_{\mu \in \partial_{\dot q} \RR (\HH(q)(t),0)} \|\eta-\mu\|_{2}^2
\end{equation*} 
for any $\eta \in \lo$.
In order to compute $\partial_{\dot q} \RR (\HH(q)(t),0)$, we use the definition of the convex subdifferential and the fact that 
$\RR (\HH(q)(t),0)=0$, from which we deduce
\[\partial_{\dot q} \RR (\HH(q)(t),0)=\{\mu \in \lo | \,  (\mu,v)_{\lo} \leq \RR(\HH(q)(t),v) \quad \forall\,v \in \lo\}.\]

Now, in view of \eqref{eq:R} combined with the fundamental lemma of the calculus of variations we have 
\[\partial_{\dot q} \RR (\HH(q)(t),0)=\{\mu \in \lo | \,  \mu \leq f(\HH(q)(t))\   \ae \text{ in }\Omega \}.\]
This means that $P_{\partial_{\dot q} \RR (\HH(q)(t),0)}(\eta)=\min\{\eta,f(\HH(q)(t))\}$ and since $\eta -\min\{\eta,f(\HH(q)(t))\}=\max \{\eta-f(\HH(q)(t)),0\}$ we can finally write \eqref{ode} as 
   \begin{equation}\label{F}
\dot q(t)=\frac{1}{\epsilon} \max \{g(q(t),\ell(t))-f(\HH(q)(t)),0\} \quad \ae \text{ in } (0,T).
 \end{equation}To summarize, we have shown that the evolution in \eqref{eq:ev0} is equivalent to  \eqref{F}.

 To solve \eqref{F}, we apply a fixed-point argument. For this, we take a look at the 
 mapping $L^2(0,t;\lo) \ni \eta \mapsto \GG(\eta) \in H^1(0,t;\lo)$, given by 
\begin{equation*}\label{eq:FF}
\GG(\eta)(\tau):=\int_0^\tau  \max (g(\eta(s),\ell(s))-(f  \circ \HH)(\eta)(s)) \;ds \quad \forall\, \tau \in [0,t],
\end{equation*}
where $t \in (0,T]$ is to be determined so that $\GG:L^2(0,t;\lo) \to L^2(0,t;\lo)$ is a contraction. 
For all $q_1,q_2 \in L^2(0,t; \lo)$ the following estimate is true
\begin{equation}\label{eqq}
\begin{aligned}
 \|\GG(q_1)(\tau)&-\GG(q_2)(\tau)\|_{\lo} \leq  \int_0^\tau  \|g(q_1(s),\ell(s))-g(q_2(s),\ell(s))\|_{\lo} \,ds
 \\&\quad + \int_0^\tau \|(f \circ \HH)(q_1)(s)- (f \circ \HH)(q_2)(s)\|_{\lo} \;ds
\\ &\leq c\,  \int_0^\tau \|q_1(s)-q_2(s) \|_{\lo} \,ds+L_f    \, L_{\HH}\int_0^\tau \int_0^s \|q_1(\zeta)-q_2(\zeta) \|_{\lo} \,d\zeta \;ds
\\ &\leq c \,t^{1/2} \|q_1-q_2\|_{L^2(0,t;\lo)} + tL_f  \, L_{\HH}\|q_1-q_2\|_{L^1(0,t;\lo)}
\\ &\leq (c \,t^{1/2} + L_f\, L_{\HH} \,t^{3/2}) \|q_1-q_2\|_{L^2(0,t;\lo)}
\quad  \text{ for all }\tau \in [0,t],
\end{aligned}\end{equation}where $c>0$ is a positive constant. Here we used the fact that $\max:\lo \to \lo$ is Lipschitzian with constant $1$, the definition of $g$ (see \eqref{g}) combined with the boundedness of $\phi$, and the estimate \eqref{eq:f_h}. From \eqref{eqq} we deduce
 \begin{equation}\label{eq:kontrak}
 \|\GG(q_1)-\GG(q_2)\|_{L^2(0,t; \lo)}
\leq (c\,t + L_f \, L_{\HH}\, t^{2}) \|q_1-q_2\|_{L^2(0,t;\lo)}.\end{equation}which allows us to conclude that $\frac{1}{\epsilon}\, \GG$ is a contraction for a small enough $t$. Thus, the PDE \eqref{F} restricted on $(0,t)$ admits a unique solution in $H^1_0(0,t;\lo)$(see e.g.\,\cite[Thm.\, 7.2.3]{emm}).
Now, the unique solvability of \eqref{F} on the whole interval $(0,T)$ and the desired regularity of $q$ follows by a concatenation argument.

Finally, we recall that  $\varphi(\cdot)=\phi(q(\cdot),\ell(\cdot))$, cf.\ \eqref{eq:ell_eq} and  we deduce from \eqref{eq:ell}  that     $\varphi \in \llu$. To summarize, we obtained that \eqref{eq:q} admits a unique solution $(q,\varphi)\in  H^1_0(0,T;\lo) \times \llu $, which, owing to \eqref{eq:ell_eq} and \eqref{F}, is characterized by  \eqref{eq:syst_diff}.
\end{proof}

\begin{lemma}\label{lem:lip_S}The solution map associated to \eqref{eq:q} $$S: \llU \ni \ell \mapsto (q,\varphi) \in   \hhyn \times \llu$$ is Lipschitz continuous.
\end{lemma}
\begin{proof}
Let $\ell_1,\ell_2 \in \llU$ be arbitrary, but fixed. In the following, we abbreviate $(q_i,\varphi_i):=S(\ell_i)$ and $g(q_i(\cdot),\ell_i(\cdot)):=\beta(\phi(q_i(\cdot),\ell_i(\cdot))-q_i(\cdot)),\ i=1,2$, where $\phi$ is the solution operator of \eqref{eq:ell} . In view of Proposition \ref{lem:ode} combined with \eqref{eq:f_h}, we obtain
\begin{align*}
\|(q_1-q_2)(t)\|_{\lo}  &\leq  \frac{1}{\epsilon} \int_0^t     \|g(q_1(s),\ell_1(s))-g(q_2(s),\ell_2(s))\|_{\lo} \,ds
 \\&\quad +  \frac{1}{\epsilon} \int_0^t \|(f \circ \HH)(q_1)(s)-(f \circ \HH)(q_2)(s)\|_{\lo} \;ds
\\ &\leq c\,  \int_0^t \|q_1(s)-q_2(s) \|_{\lo} +\|\ell_1(s)-\ell_2(s) \|_{\hoo} \,ds
\\&\quad +\frac{1}{\epsilon}L_f    \, L_{\HH}\int_0^t \int_0^s \|q_1(\zeta)-q_2(\zeta) \|_{\lo} \,d\zeta \;ds \quad \forall\,t \in [0,T],
\end{align*}where $c>0$ is a constant dependent only on the given data. Then, applying  Gronwall's inequality leads to 
$$
\|(q_1-q_2)(t)\|_{\lo} \leq \hat c   \int_0^t \|\ell_1(s)-\ell_2(s) \|_{\hoo} \,ds \quad \forall\,t \in [0,T],
$$where $\hat c>0$ is a constant dependent only on the given data. By employing again \eqref{eq:ode} and by estimating as above without integrating over time, we obtain 
\begin{equation}\label{q_lip}
\|q_1-q_2\|_{H^1(0,T;\lo)} \leq \widetilde c \, \|\ell_1-\ell_2 \|_{\llU},
\end{equation}where $\widetilde c>0$ is another constant dependent only on the given data. 
Now, the desired result follows from  $\varphi_i=\phi(q_i, \ell_i),\ i=1,2$, $\phi \in \LL(\lo \times \hoo,\ho)$ and \eqref{q_lip}.
\end{proof}

\begin{lemma}\label{lem:f_h}
The mapping $\F:\llz \to \llz$ is Hadamard directionally differentiable with
 \begin{equation}\label{eq:f_h_deriv}
\F'(\eta;\delta \eta)=f'(\HH(\eta);\HH'(\eta)(\delta \eta)) \quad \forall\,\eta,\delta \eta \in \llz.
\end{equation}Moreover, for all $\eta, \delta \eta_1, \delta \eta_2 \in \llz$, it holds
\begin{equation}\label{eq:h'}
\|\F'(\eta;\delta \eta_1)(t)-\F'(\eta;\delta \eta_2)(t)\|_{\lo)} \leq L_f\,L_{\HH}  \int_0^t \|\delta \eta_1(s)-\delta \eta_2(s)\|_{\lo} \,ds
\end{equation} \ae \text{in }$(0,T)$.
\end{lemma}
\begin{proof}
In view of the differentiability properties of $\HH$ and $f$, the mapping $\F:\llz \to \llz$ is 
Hadamard directionally differentiable  \cite[Def.\ 3.1.1, Lem.\ 3.1.2(b)]{schirotzek}. To see this, we first note that $f:\llz \to \llz$ is Hadamard directionally differentiable, since it is directionally differentiable (by Assumption \ref{assu:stand}.\ref{it:stand2} and Lebesgue's dominated convergence theorem, see e.g.\, \cite[Lemma A.1]{susu_dam}) and Lipschitz-continuous.  In view of Assumption \ref{assu:stand}.\ref{it:stand1}, chain rule \cite[Prop.\ 3.6(i)]{shapiro} implies that $\F$ is Hadamard directionally differentiable as well, with directional 
derivative given by \eqref{eq:f_h_deriv}. To prove \eqref{eq:h'}, we observe that, 
as a consequence of \eqref{eq:f_h}, we  have 
$$\frac{1}{\tau}\|\F (\eta+ \tau \delta \eta_1)(t)-\F (\eta+\tau \delta \eta_2)(t)\|_{\lo} \leq  L_f\, L_\HH\, \int_0^t \|\delta \eta_1(s)-\delta \eta_2(s)\|_{\lo} \,ds $$$\ae $in $(0,T)$, for all  $\eta, \delta \eta_1,\delta \eta_2 \in \llz$ and all $\tau>0$. Passing to the limit $\tau \searrow 0$, where one uses the  directional differentiability of $f \circ \HH$ and  the fact that convergence in $\llz$ implies $\ae$ convergence in $\lo$ for a subsequence,
then  yields  
the desired estimate.
\end{proof}

\begin{proposition}[Directional differentiability]\label{lem:rd}
The operator $S: \llU \to   \hhyn \times \llu$ is  directionally differentiable. 
Its directional derivative $(\dd,\df):=S'(\ell;\dl)$ at $\ell  \in \llU$ in direction $\dl \in \llU$ is the unique solution of
 \begin{subequations}\label{eq:ode_lin_q}\begin{gather}
\dot \dd(t)=  \frac{1}{\epsilon} \maxlim ' \big(z(t); -\beta(\dd(t)-\df(t))-f'(\HH(q);\HH'(q)(\delta q))(t)\big)\ \text{ in }\lo, \ \, \dd(0)=0,\label{eq:ode_lin_q1} 
\\  - \alpha \Delta \df(t) + \beta \,\df(t) = \beta \dd(t) +\dl(t) \quad  \text{in }H^{1}(\Omega)^\ast \label{eq:ode_lin_q2}
\end{gather}
 \end{subequations}a.e.\ in $(0,T)$, where we abbreviate $ z(t):=-\beta(q(t)-\varphi(t))-(f \circ \HH)(q)(t)$.   \end{proposition}
\begin{proof}
We start by examining the solvability of \eqref{eq:ode_lin_q}. To this end,  we just check that the 
 mapping $L^2(0,t;\lo) \ni \eta \mapsto \widehat \GG(\eta) \in H^1(0,t;\lo)$, given by 
\begin{equation*}\label{eq:FF1}
\widehat \GG(\eta)(\tau):=\int_0^\tau  \maxlim ' \big(z(s); -\beta(\eta(s)-\phi(\eta(s), \dl(s))-f' \big(\HH(q);\HH'(q)(\eta)\big)(s)\big) \;ds
\end{equation*}for all $\tau \in [0,t]$, is Lipschitzian from $L^2(0,t;\lo)$ to $L^2(0,t;\lo)$ with constant smaller than $\epsilon$, for 
$t \in (0,T]$ small enough. Then, by using the arguments employed at the end of  the  proof of Proposition \ref{lem:ode}, we can deduce that, for any $\dl \in \llU$, \eqref{eq:ode_lin_q} admits a unique solution $(\dd, \df ) \in \hhyn \times \llu$.  
For all $\eta_1,\eta_2 \in L^2(0,t; \lo)$ the following estimate is true
\begin{align*}
 \|\widehat \GG(\eta_1)(\tau)&-\widehat \GG(\eta_2)(\tau)\|_{\lo} \leq  \int_0^\tau  \|g(\eta_1(s),\dl(s))-g(\eta_2(s),\dl(s))\|_{\lo} \,ds
 \\&\quad + \int_0^\tau \|f'( \HH(q);\HH'(q;\eta_1))(s)-f'( \HH(q);\HH'(q;\eta_2))(s)\|_{\lo} \;ds
\\ &\leq c\,  \int_0^\tau \|\eta_1(s)-\eta_2(s) \|_{\lo} \,ds+L_f    \, L_{\HH}\int_0^\tau \int_0^s \|\eta_1(\zeta)-\eta_2(\zeta) \|_{\lo} \,d\zeta \;ds
\\ &\leq c \,t^{1/2} \|\eta_1-\eta_2\|_{L^2(0,t;\lo)} + tL_f  \, L_{\HH}\|\eta_1-\eta_2\|_{L^1(0,t;\lo)}
\\ &\leq (c \,t^{1/2} + L_f\, L_{\HH} \,t^{3/2}) \|\eta_1-\eta_2\|_{L^2(0,t;\lo)}
\quad  \text{ for all }\tau \in [0,t],
\end{align*}where $c>0$ is a positive constant; note that here we abbreviated again $g(\eta_i(\cdot),\dl(\cdot)):=\beta(\phi(\eta_i(\cdot),\dl(\cdot))-\eta_i(\cdot)),\ i=1,2$. Here we used the fact that $\max'(z(s), \cdot):\lo \to \lo$ is Lipschitzian with constant $1$, the  boundedness of $\phi$ (see \eqref{eq:ell}), and \eqref{eq:h'} in combination with \eqref{eq:f_h_deriv}. Then, we obtain an estimate similar to \eqref{eq:kontrak} which allows us to conclude the fact that $\frac{1}{\epsilon} \widehat \GG$ is a contraction.

Next we focus on the convergence of the difference quotients associated with the mapping $S$.
We begin by observing that the operator $\max:\llz \to \llz$  is Hadamard directionally differentiable  \cite[Def.\ 3.1.1, Lem.\ 3.1.2(b)]{schirotzek}, since it is directionally differentiable (by Lebesgue's dominated convergence theorem, see e.g.\ \cite[Lem.\ A.1]{susu_dam}) and Lipschitz-continuous.  
  Moreover, 
\begin{equation*}\label{G}
G:(\eta,\psi) \mapsto -\beta(\eta-\phi(\eta,\psi))-(f \circ \HH)(\eta)\end{equation*} is directionally differentiable from $\llz \times \llU$ to $\llz$, since $\phi$ is linear and bounded between these spaces (cf.\ \eqref{eq:ell}) and  as a result of Lemma \ref{lem:f_h}. Now chain rule \cite[Prop.\ 3.6(i)]{shapiro} implies that $$\FF:=\max \circ \, G$$ is (Hadamard) directionally differentiable from $L^2(0,T;\lo) \times \llU$ to  $L^2(0,T;\lo)$ with
 \begin{equation*}\label{eq:g_d}
  \FF'((q,\ell);(\dd, \dl))=\maxlim '\big(G(q,\ell);G'((q,\ell);(\dd, \dl))\big)   \end{equation*}
  for all $(q,\ell),(\dd, \dl) \in  L^2(0,T;\lo) \times \llU$.
For simplicity, in the following we abbreviate $q^{\tau}:=S_1(\ell + \tau \,\dl)$, where $\tau>0$ is arbitrary, but fixed. $S_1$ denotes the first component of the map $S$, i.e.,\ $S_1: \llU \to \hhyn$ is the solution map associated with \eqref{F}.
By combining the equations for $q^{\tau}$, $q$ and \eqref{eq:ode_lin_q}, we obtain
  \begin{equation}\label{eq:www} \begin{aligned}
 \frac{d}{dt} \Big(\frac{q^\tau-q}{\tau}-\delta q \Big)&=\frac{\FF(q^\tau,\ell + \tau \,\dl) -\FF(q,\ell) }{\tau}-\FF'\big((\yy,\ell);(\delta q,\dl)\big) \quad \text{\ae in } (0,T),  
 \\  \Big(\frac{q^\tau-q}{\tau}-\delta q\Big)(0) &= 0.
 \end{aligned}
 \end{equation}This implies  
   \begin{equation}\label{eq:awalin-}
 \begin{aligned}
&\Big\|\Big(\frac{q^\tau-q}{\tau}-\delta q \Big)(t)\Big\|_{\lo}
\\  &\  \leq \int_0^t \Big\|\frac{\FF \big (q^\tau,\ell+ \tau \,\dl \big)(s) -\FF\big((q, \ell )+\tau (\delta q,\dl) \big) (s) }{\tau}\Big\|_{\lo}
\\ &\qquad + \Big\| \underbrace{ \frac{\FF\big((q, \ell )+\tau (\delta q,\dl) \big) (s)-\FF(q,\ell )(s)}{\tau}
 -\FF'\big((\yy,\ell);(\delta q,\dl)\big)(s)  }_{=:A_\tau(s)}  \Big\|_{\lo} \,ds
  \\  &\  \leq \int_0^t  \Big\|\frac{G(q^\tau, \ell +\tau \dl)(s)-G((q, \ell )+\tau (\delta q,\dl) \big) (s)}{\tau}\Big\|_{{\lo}} \,ds+ \|A_\tau\|_{L^1(0,t;\lo)}
 \\  &\  \leq c\, \int_0^t   \Big\|\Big(\frac{q^\tau-q}{\tau}-\delta q \Big)(s)\Big\|_{\lo}  \,ds+L_f    \, L_{\HH}\int_0^t   \int_0^s \Big\|\Big(\frac{q^\tau-q}{\tau}-\delta q \Big)(\zeta)\Big\|_{\lo}  \,d\zeta \,ds\\&\qquad+ \|A_\tau\|_{L^1(0,T;\lo)} \qquad \forall\,t \in [0,T],
 \end{aligned}
 \end{equation}where $c>0$ is the positive constant appearing in \eqref{eqq}. In \eqref{eq:awalin-} we  used again the Lipschitz continuity of $\max:\lo \to \lo$, the boundedness of $\phi$ (cf.\,\eqref{eq:ell_eq} and \eqref{eq:ell}), and the estimate \eqref{eq:f_h}. Applying Gronwall's inequality in \eqref{eq:awalin-}  yields 
   \begin{equation}\label{eq:wwww}
\Big\|\Big(\frac{q^\tau-q}{\tau}-\delta q \Big)(t)\Big\|_{\lo} \leq C\,\|A_\tau\|_{L^1(0,T;\lo)} \qquad \forall\,t \in [0,T],
 \end{equation}where $C>0$ is a constant dependent only on the given data. Now, \eqref{eq:www}  and estimating  as in \eqref{eq:awalin-}, in combination with \eqref{eq:wwww}, leads to    \begin{equation}\label{eq:ww}
\Big\|\frac{q^\tau-q}{\tau}-\delta q \Big\|_{H^1(0,T;{\lo})} \leq \widehat C\,\|A_\tau\|_{L^2(0,T;\lo)} \quad \forall\, \tau >0,
 \end{equation}where $\hat C>0$ is a constant dependent only on the given data. On the other hand, we recall the definition of $A_\tau$ in  \eqref{eq:awalin-} and the fact that $\FF$ is directionally differentiable from $L^2(0,T;\lo) \times \llU$ to  $L^2(0,T;\lo)$, which 
 implies$$\|A_\tau\|_{L^2(0,T;\lo)} \to 0 \quad \text{as } \tau \searrow 0.$$
In view of \eqref{eq:ww}, we have shown that $S_1:\llU \to H^1(0,T;{\lo}$ is directionally differentiable with $S_1'(\ell;\dl)=\dd$.    Further, from  \eqref{eq:ell_eq} we have $ S_2(\ell)=\phi(S_1(\ell),\ell)$ for all $\ell \in \llU$, where $S_2$ is the second component of the operator $S$, i.e., $S_2: \llU \ni \ell \mapsto \varphi \in   \llu$. 
 Thus, $S_2$ is directionally differentiable as well, since $\phi \in \LL(\lo \times \hoo ; \ho)$ and $S_1$  is directionally differentiable. Its    directional derivative $S_2'(\ell;\dl)$ is given by
$\phi(S_1'(\ell;\dl),\dl)$, i.e., $S_2'(\ell;\dl)=\df$, see  \eqref{eq:ode_lin_q}. The proof is now complete.
\end{proof}

\section{The optimal control problem}\label{sec:o}
Now, we turn our attention to the optimal control of the fatigue damage model \eqref{eq:q}. In the remaining of the paper, we are concerned with the examination of the following optimal control problem
\begin{equation*}
 \begin{aligned}
  \min_{\ell \in H^1(0,T;\lo)} \quad & J(q,\varphi,\ell)
  \\
 \text{s.t.} \quad & (q,\varphi) \text{ solves }\eqref{eq:q} \text{ with r.h.s.\ }\ell.
 \end{aligned}
\end{equation*}
In view of Proposition \ref{lem:ode}, this can also be formulated as 
\begin{equation}\tag{P}\label{eq:optt}
 \left.
 \begin{aligned}
  \min_{\ell \in H^1(0,T;\lo)} \quad & J(q,\varphi,\ell)
  \\
 \text{s.t.} \quad & (q,\varphi) \text{ solves }\eqref{eq:syst_diff} \text{ with r.h.s.\ }\ell.
 \end{aligned}
 \quad \right\}
\end{equation}

\begin{assumption}\label{assu:j_ex}
  The functional $J$ satisfies
  $$J(q,\varphi,\ell)=j(q,\varphi)+\frac{1}{2} \|\ell\|^2_{H^1(0,T;\lo)},$$
where  $j: L^2(0,T;\lo)\times \llu \to \R$ 
  is continuously Fr\'echet-differentiable.
\end{assumption}

Note that Assumption \ref{assu:j_ex} is satisfied by classical objectives of tracking type such as
\begin{equation*}
J_{ex}( q ,\varphi, \ell) := \frac{1}{2}\, \|q-q_d\|^2_{L^2(0,T;L^2(\Omega))}
 + \frac{\kappa}{2} \|\varphi\|^2_{L^2(0,T;H^1(\Omega))}+\frac{1}{2} \|\ell\|^2_{H^1(0,T;\lo)},
\end{equation*}
where $q_d \in L^2(0,T;L^2(\Omega))$ and $\kappa \geq 0$.

\begin{proposition}[Existence of optimal solutions for \eqref{eq:optt}]\label{prop:ex}
The optimal control problem \eqref{eq:optt} admits at least one solution in $\hs$.
\end{proposition}
\begin{proof}
The assertion follows by standard arguments which rely on the direct method of the calculus of variations combined with the radial unboundedness of the reduced objective $$\hs \ni \ell \mapsto J(S(\ell),\ell) \in \R,$$ the Lipschitz continuity of $S$ on $\llU$ (Lemma \ref{lem:lip_S}), the compact embedding $\hs \embed \embed \llU$ and the continuity of $j$ from Assumption \ref{assu:j_ex}.
\end{proof}

\subsection{Regularization and passage to the limit}\label{sec:reg}
In this section, we are concerned with the  derivation of a first optimality system for local optima of \eqref{eq:optt}. Based thereon, we shall improve our optimality conditions in the next section. 

To obtain a first  strong optimality system, see \eqref{eq:adjjj} below, we need the following rather non-restrictive assumption:
\begin{assumption}\label{assu:stand_eps}
In addition to Assumption \ref{assu:stand}, we require that 
the mappings associated with fatigue  in \eqref{eq:q} satisfy:
 \begin{enumerate}
  \item\label{it:s1}  The \textit{history operator} $\HH:L^2(0,T;\li) \to L^2(0,T;\li)$ fulfills 
$$\|\HH (q_1)(t)-\HH (q_2)(t)\|_{\li} \leq \widehat L_\HH\, \int_0^t \|q_1(s)-q_2(s)\|_{\li} \,ds \quad \ae \text{in } (0,T),$$for all  $q_1,q_2 \in L^2(0,T;\li)$, where $\widehat L_\HH>0$ is a positive constant.  
\item\label{it:s2}
The non-differentiable function  $f: \R \to \R$ is assumed to have one non-smooth point $n_f$.
 \end{enumerate}
\end{assumption}
{\begin{remark}
Similarly to Remark \ref{rem:volt}, we observe that Assumption \ref{assu:stand_eps}.\ref{it:s1} is satisfied by classical Volterra operators which are employed in the study of history-dependent evolutionary variational inequalities, i.e., $\HH:L^2(0,T;\li) \to L^2(0,T;\li)$
\begin{equation*}
 [0,T] \ni t \mapsto \HH(q)(t):=\int_0^t A(t-s)q(s) \,ds +q_0 \in \li,\end{equation*}
where $A\in C([0,T]; \LL(\li,\li))$ and $q_0\in \li.$

We underline that Assumption \ref{assu:stand_eps}.\ref{it:s2} is very reasonable from the point of view of applications, since fatigue degradation functions have at most two kink points   in practice \cite[Sec.\,2.6.2]{alessi1}. However, our mathematical analysis can be carried on in an analogous way if $f$ has a countable number of non-smooth points; since this is rather uncommon in applications and for the sake of a better overview, we stick to the setting where $f$ has a single non-differentiable point.
\end{remark}}
\begin{assumption}[Regularization of $f$]\label{assu:f_eps}
 For every $\varepsilon>0$, there exists a continuously differentiable function $f_\varepsilon:\R \to \R$ such that
 \begin{enumerate}
  \item\label{it:feps1} There exists a constant $C>0$, independent of $\eps$, such that 
  \begin{equation*}
|f_\varepsilon(v) - f(v)| \leq C\eps \quad \forall\,v \in \R. 
  \end{equation*}
  \item\label{it:feps2}  $f_\eps$ is Lipschitz continuous   with Lipschitz constant $\hat L_f>0$ independent of $\varepsilon$.
  \item \label{assu:f_epss} for every $\delta >0$, the sequence $\{f_\eps'\}$ converges uniformly towards $f'$ on $(-\infty, n_f-\delta] \cup [n_f+\delta,\infty)$ as $\varepsilon \searrow 0$.
  \end{enumerate}
\end{assumption}

As an immediate consequence of Assumptions \ref{assu:f_eps}.\ref{it:feps1}, we have 
  \begin{equation}\label{conv_f_eps}
\|f_\varepsilon(\eta) - f(\eta)\|_{L^\infty(0,T;\li)} \to 0 \text{ as } \varepsilon \searrow 0, \quad \forall\, \eta \in L^2(0,T;\lo).  \end{equation}

\begin{remark}If the fatigue degradation function $f$ is piecewise continuously differentiable, which is always the case in applications \cite[Sec.\,2.6.2]{alessi1}, then Assumption \ref{assu:f_eps} is fulfilled. To see this, one defines $f_\eps:=\Phi_\eps \star f$, where $\Phi_\eps$ is a standard mollifier. Then, Assumption \ref{assu:f_eps}.\ref{it:feps1} can be easily checked, see e.g.\,the proof of \cite[Thm.\,2.4]{kaballo}. Note that it is  natural that the Lipschitz continuity of  the non-linearity $f$ carries over to its regularized counterparts  with constant independent of $\varepsilon$  \cite[Chp.\ I.3.3]{tiba}. We also observe that, since $f'$ is continuous on $(-\infty, n_f-\delta] \cup [n_f+\delta,\infty)$, $f_\eps'=\Phi_\eps \star f'$ converges uniformly towards $f'$ on this interval, so that Assumption \ref{assu:f_eps}.\ref{assu:f_epss} is satisfied as well. 
\end{remark}

In the rest of the paper, we will tacitly assume that, in addition to  Assumptions \ref{assu:stand} and \ref{assu:j_ex}, Assumptions \ref{assu:stand_eps} and \ref{assu:f_eps}  are   always fulfilled, without mentioning them every time.

For an arbitrary local minimizer $\ll$ of \eqref{eq:optt}, consider the following regularization, also known as "adapted penalization", see e.g.\ \cite{barbu}:

\begin{equation}\label{eq:p_eps}\tag{P$_\varepsilon$}
 \left.
 \begin{aligned}
\min_{\ell \in H^1(0,T;\lo)} &\quad  J(q,\varphi,\ell)+\frac{1}{2}\|\ell-\ll\|_{H^1(0,T;\lo)}^2\\
  \text{s.t.} &\quad 
   \begin{aligned}[t]
& \dot q(t)   =   \frac{1}{\epsilon} \maxeps (-\beta(q(t)-\varphi(t))-(f_\eps \circ \HH)(q)(t))\ \text{ in }\lo, \quad q(0) = 0,   
\\&- \alpha \Delta \varphi(t) + \beta \,\varphi(t)   
= \beta q(t) +\ell(t) \quad  \text{in }H^{1}(\Omega)^\ast, \quad \text{a.e.\ in }(0,T), 
\end{aligned}
 \end{aligned}
 \quad \right\}
\end{equation}where \vspace{-0.4cm}  \[ 
\maxlim{_\varepsilon}:\R \to \R, \quad  \maxlim{_\varepsilon}(x) := 
 \begin{cases}
  0, & x \leq 0,\\
  \frac{1}{2\varepsilon} \, x^2, & x\in \, (0,\varepsilon)\, ,\\
  x - \frac{\varepsilon}{2}, & x \geq \varepsilon.
 \end{cases}
\]

\begin{lemma}\label{lem:loc}
For each local optimum $\ll$ of \eqref{eq:optt} there exists a sequence of local minimizers $\{\ell_\eps\}$ of 
\eqref{eq:p_eps}
such that 
 \begin{equation}\label{l_conv}
\ell_{\eps} \to \ll \quad  \text{in }H^1(0,T;\lo) \quad \text{as }\eps \searrow 0. \end{equation}
Moreover, 
\begin{equation}\label{y_conv}
  S_\varepsilon(\ell_{\varepsilon}) \to S(\ll)  \quad \text{in } \hhyn \times \llu \quad \text{as }\eps \searrow 0,
  \end{equation}where $S_\varepsilon:\llU \ni \ell \mapsto (q_\eps,\varphi_\eps) \in   \hhyn \times \llu$ is the control-to-state map associated to the state equation in \eqref{eq:p_eps}.\end{lemma}

\begin{proof}see Appendix \ref{sec:a}.
\end{proof}

The next result  is essential for the solvability of the first adjoint equation in \eqref{eq:adjjj}.

\begin{lemma}For all $\eta, \delta \eta \in \llz$ it holds
\begin{equation}\label{eq:h_adj}
\|[(f_\eps \circ \HH)'(\eta)]^\star (\delta \eta)(t)\|_{\lo}
\leq \hat L_f\,L_{\HH}  \int_t^T \|\delta \eta(s)\|_{\lo} \,ds \ \quad \ae \text{ in }(0,T),
\end{equation}where $[(f_\eps \circ \HH)'(\eta)]^\star:\llz \to \llz$
 stands for the adjoint operator of $(f_\eps \circ \HH)'(\eta)$.
 \end{lemma}\begin{proof}
Let $\psi \in \llz$ be arbitrary, but fixed. By virtue of \eqref{eq:h'} (applied for $f_\eps$ instead of $f$), we have
\begin{equation*}\begin{split}
([(f_\eps \circ \HH)'(\eta)]^\star (\delta \eta),\psi)_{\llz}
=((f_\eps \circ \HH)'(\eta)(\psi),\delta \eta)_{\llz} \\ \leq \int_0^T \hat L_f\,L_{\HH}  \int_0^T  \raisebox{3pt}{$\chi$}_{[0,t]}(s) \|\psi(s)\|_{\lo}  \,ds \  \|\delta \eta(t)\|_{\lo}\,dt
\\= \hat  L_f\,L_{\HH} \int_0^T \int_0^T  \raisebox{3pt}{$\chi$}_{[0,s]}(t)\|\delta \eta(s)\|_{\lo} \,ds \  \| \psi(t)\|_{\lo} \,dt
\\=\hat  L_f\,L_{\HH} \int_0^T \int_t^T \|\delta \eta(s)\|_{\lo} \,ds \  \| \psi(t)\|_{\lo} \,dt.
\end{split}\end{equation*}Note that in the first identity  we made use of Fubini's theorem. Now, testing with $\psi:=v \rho$, where $v \in \lo$ and $\rho \in L^2(0,T)$, $\rho \geq 0$, are arbitrary, but fixed yields
\begin{equation*}
\int_0^T ([(f_\eps \circ \HH)'(\eta)]^\star (\delta \eta)(t),v )_{\lo} \rho(t) \,dt
\leq \hat  L_f\,L_{\HH} \int_0^T \int_t^T \|\delta \eta(s)\|_{\lo} \,ds \  \| v\|_{\lo} \rho(t)\,dt.
\end{equation*}Applying the fundamental lemma of the calculus of variations then gives in turn
\begin{equation*}
([(f_\eps \circ \HH)'(\eta)]^\star (\delta \eta)(t),v )_{\lo} 
\leq \hat  L_f\,L_{\HH}  \int_t^T \|\delta \eta(s)\|_{\lo} \,ds \  \| v\|_{\lo} 
\end{equation*}
a.e.\ in $(0,T)$. Since $v \in \lo$ was arbitrary, the proof is now complete.
\end{proof}

To show that the relations in \eqref{eq:add} below are valid, we need to prove that the convergence in \eqref{y_conv} is true in $L^\infty(0,T;\li)$ as well. This is confirmed by the following

\begin{lemma}\label{lem:loc2}
Let $\{\ell_{\eps}\}$ be the sequence of local minimizers from Lemma \ref{lem:loc} associated to  a local optimum  $\ll$ of \eqref{eq:optt}. Then,
\begin{equation}\label{y_conv2}
  S_\varepsilon(\ell_{\varepsilon}) \to S(\ll)  \quad \text{in } L^\infty(0,T;\li) \times L^\infty(0,T;\li)\quad \text{as }\eps \searrow 0.
  \end{equation}
  \end{lemma}
\begin{proof}
Let us first show that $ ( \bar q,\bar \varphi)$  belongs to $L^\infty(0,T;\li) \times L^\infty(0,T;\li) $. The assertion for $ ( q_\eps, \varphi_\eps)$ follows in a complete analogous way. By taking a look at \eqref{eq:syst_diff}, we see that, since $\ll \in L^\infty(0,T;\lo)$, the mapping $\bar \varphi $ belongs to $ L^\infty(0,T;\li) $; this follows by the so-called Stampacchia method, cf.\ e.g. \cite[Chp. 7.2.2]{troe}. Then, by arguing as in the proof of Proposition \ref{lem:ode}, where one employs Assumption \ref{assu:stand_eps}.\ref{it:s1}, one obtains that $\bar q \in H^1(0,T;\li) \subset L^\infty(0,T;\li)$. Now, to show the convergence \eqref{y_conv2}, we subtract the equation associated to $\bar q$ (see \eqref{eq:ode}) from the one associated to $q_\eps$ (see \eqref{eq:syst_diff11_e}). By using the fact that $|\maxlim{_\varepsilon}(x)-\max(x)| \leq \eps \ \forall\, x \in \R$, and by relying on the Lipschitz continuity of $\max$  and $f$, as well as Assumptions \ref{assu:f_eps}.\ref{it:feps1}, we arrive at 
\begin{align*}
\|(q_\eps-\bar q)(t)\|_{\li} &\leq  2\eps t +c\,  \int_0^t \|q_\eps(s)-\bar q (s) \|_{\li} +\|\varphi_\eps(s)-\bar \varphi (s) \|_{\li}\\&\quad +  L_f  \int_0^t  \|\HH(q_\eps)(s)-\HH(\bar q)(s) \|_{\li} \;ds 
\\
 &\leq  2\eps t +c\,  \int_0^t \|q_\eps(s)-\bar q (s) \|_{\li} +\|\ell_\eps(s)-\bar \ell (s) \|_{\lo}\\&\quad +  L_f    \, \widehat L_{\HH}\int_0^t \int_0^s \|q_\eps(\zeta)-\bar q(\zeta) \|_{\li} \,d\zeta \;ds 
\quad \forall\,t \in [0,T],
\end{align*}where $c>0$ is a constant dependent only on the given data; note that in the last inequality we used Assumption \ref{assu:stand_eps}.\ref{it:s1}. Then, applying  Gronwall's inequality leads to 
$$
\|(q_\eps-\bar q)(t)\|_{\li} \leq 2\eps t + \hat c   \int_0^t \|\ell_\eps(s)-\bar \ell (s) \|_{\lo} \,ds \quad \forall\,t \in [0,T],
$$where $\hat c>0$ is a constant dependent only on the given data. By employing \eqref{l_conv}, we can finally deduce that $
 q_\eps \to \bar q  \  \text{in } L^\infty(0,T;\li) .
$ In view of \eqref{eq:el}, the proof is now complete.
\end{proof}

We are now in the position to state the main result of this subsection.
\begin{proposition}\label{lem:reg}
Suppose that Assumptions \ref{assu:j_ex},  \ref{assu:stand_eps} and  \ref{assu:f_eps} are fulfilled. Let $\ll$ be a local optimum  of \eqref{eq:optt}  with associated state $(\bar q,\bar \varphi) \in H^1_0(0,T;\lo) \times L^2(0,T;\ho)$. Then there exist  adjoint states $$\xi \in H^1_T(0,T;\lo) \text{ and }w \in L^2(0,T;\ho)$$  and  multipliers $\lambda \in L^\infty(0,T;\lo) \text{ and }\mu \in L^\infty(0,T;\lo)$ such that the following optimality system is satisfied
\begin{subequations}\label{eq:adjjj}\begin{gather}
-\dot{\xi}-\beta \big(w-\lambda\big)+\HH'(\bar q)^\star (\mu)=\partial_q j(\bar  q, \bar  \varphi) \ \text{ in }\llz,\quad \xi(T)=0,\label{eq:adj_s1}
\\-\alpha  \Delta w  + \beta \big(w-\lambda\big)=\partial_\varphi j(\bar  q, \bar  \varphi) \ \text{ in }L^2(0,T;\hoo),\label{eq:adj_s2}
\\\lambda(t,x)=\frac{1}{\epsilon} \raisebox{3pt}{$\chi$}_{\{ \bar  z>0\}}(t,x) \xi(t,x) \quad \text{a.e.\ where  }\bar z(t,x) \neq 0,\label{eq:lambda1}
\\ \mu(t,x) =f'(\HH(\bar q)(t,x)) \lambda(t,x) \quad \text{a.e.\ where  }\HH(\bar q)(t,x) \neq n_f,\label{eq:mu1}
\\(w,\dl )_{L^2(0,T;\lo)}+(\ll,\dl)_{H^1(0,T;\lo)}=0 \quad 
\forall \dl \in H^1(0,T;\lo),\label{eq:adj_s3}
\end{gather}\end{subequations}where we abbreviate $\bar z:=-\beta(\bar q-\bar \varphi)-(f \circ \HH)(\bar q)$. 
\end{proposition}
\begin{proof}Let $\{\ell_{\eps}\}$ be the sequence of local minimizers from Lemma \ref{lem:loc}. Since  $\ell_{\eps}$ is locally optimal for \eqref{eq:p_eps} and on account of the differentiability properties of $S_{\eps}$, cf.\ Appendix \ref{sec:a}, and $J$, see Assumption \ref{assu:j_ex}, we can write down the necessary optimality condition
 \begin{equation}\label{nec}\begin{aligned}
 j'(S_{\eps}(\ell_{\eps}))(S_{\eps}'(\ell_{\eps})(\dl))+(\ell_{\eps},\dl)_{H^1(0,T;\lo)}
+(\ell_{\eps}-\ll,\dl)_{H^1(0,T;\lo)}=0
\end{aligned}\end{equation} for all $\dl \in H^1(0,T;\lo)$.
 Now, let us consider the system  \small 
\begin{subequations}\label{eq:adjj}\begin{gather}
-\dot{\xi_{\eps}}(t)-\beta \big(w_{\eps}(t)-\frac{1}{\epsilon} \maxlim{_{\eps}}'(z_\eps(t))\xi_{\eps}(t)\big)+\HH'(q_\eps)^\star \Big(f_\eps '(\HH(q_\eps))\big(\frac{1}{\epsilon} \maxlim{_{\eps}}'(z_\eps)\xi_{\eps}\big)\Big)(t)=\partial_q j(S_{\eps}(\ell_{\eps}))(t),\quad \xi_{\eps}(T)=0, \label{eq:adj_syst1}
\\-\alpha  \Delta w_{\eps}(t)  + \beta \big(w_{\eps}(t)-\frac{1}{\epsilon}\maxlim{_{\eps}}'(z_\eps(t))\xi_{\eps}(t) \big)=\partial_\varphi j(S_{\eps}(\ell_{\eps}))(t)  \label{eq:adj_syst2}
\end{gather}\end{subequations}\normalsize a.e.\ in $(0,T)$, where we abbreviate $z_\eps:=-\beta(q_\eps-\varphi_\eps)-(f_\eps \circ \HH)(q_\eps)$ and $(q_{\eps},\varphi_{\eps}):=S_{\eps}(\ell_{\eps})$. In \eqref{eq:adj_syst1}, $\HH'(q_\eps)^\star:\llz \to \llz$ stands for the adjoint operator of $ \HH'(q_\eps)$.  

By  arguments inspired e.g.\ from the proof of \cite[Lem.\ 5.7]{susu_dam} in combination with the estimate \eqref{eq:h_adj}, one obtains that \eqref{eq:adjj} admits a unique solution $(\xi_{\eps}, w_{\eps}) \in H^1_T(0,T;\lo) \times L^2(0,T;\ho)$. Let us go a little more into detail concerning the solvability of \eqref{eq:adj_syst1}. In this context one checks if the mapping  $L^2(0,t;\lo) \ni \eta \mapsto  \GG(\eta) \in H^1(0,t;\lo)$, given by 
\begin{align*}
\GG(\eta)(\tau)&:=\int_0^\tau  \beta \big(w_{\eps}(T-s,\eta(s))-\frac{1}{\epsilon} \maxlim{_{\eps}}'(z_\eps(T-s))\eta(s)\big)
\\&-\underbrace{\HH'(q_\eps)^\star \Big(f_\eps '(\HH(q_\eps))\big(\frac{1}{\epsilon} \maxlim{_{\eps}}'(z_\eps)\eta(T-\cdot)\big)\Big)}_{=[(f_\eps \circ \HH) '(q_\eps)]^\star \big(\frac{1}{\epsilon} \maxlim{_{\eps}}'(z_\eps)(\eta(T-\cdot))\big)}(T-s)+\partial_q j(S_{\eps}(\ell_{\eps}))(T-s) \;ds
\end{align*}for all $\tau \in [0,t]$, is Lipschitzian from $L^2(0,t;\lo)$ to $L^2(0,t;\lo)$ with constant smaller than $1$, for 
$t \in (0,T]$ small enough; here $w_{\eps}(t,v)$ denotes the solution of 
$$-\alpha  \Delta w_{\eps}(t,v)  + \beta \big(w_{\eps}(t,v)-\frac{1}{\epsilon}\maxlim{_{\eps}}'(z_\eps(t))v \big)=\partial_\varphi j(S_{\eps}(\ell_{\eps}))(t)$$
for $t \in [0,T]$ and $v \in \lo.$
We observe that, for all $\eta_1,\eta_2 \in L^2(0,t; \lo)$, the following estimate is true
\begin{align*}
 \| \GG(\eta_1)(\tau)&- \GG(\eta_2)(\tau)\|_{\lo} 
\leq c\,  \int_0^\tau \|\eta_1(s)-\eta_2(s) \|_{\lo} \,ds
\\&\qquad+ \int_0^\tau \hat L_f\, L_{\HH} \, \int_{T-s}^T \|(\eta_1-\eta_2)(T-\zeta)\|_{\lo} \,d\zeta \;ds
 \\ &\leq c \,t^{1/2} \|\eta_1-\eta_2\|_{L^2(0,t;\lo)} +\hat L_f\, L_{\HH} \, \int_0^\tau \int_{0}^s \|(\eta_1-\eta_2)(\zeta)\|_{\lo} \,d\zeta \;ds
\\ &\leq (c \,t^{1/2} + \hat L_f\, L_{\HH} \,t^{3/2}) \|\eta_1-\eta_2\|_{L^2(0,t;\lo)}
\quad  \text{ for all }\tau \in [0,t],
\end{align*}
where in the first inequality we used the global Lipschitz-continuity of $\max_{\eps}$ with constant $1$ and \eqref{eq:h_adj}; now the reader is referred to  the first part of the proof of Proposition \ref{lem:rd} where the exact type of estimate was established in order to obtain that $\eta= \GG(\eta)$ admits a solution in $H_0^1(0,T;\lo)$; finally, a transformation of the variables yields that $(\xi_\eps, w_\eps):=(\eta(T-\cdot), w_\eps(t,\eta(T-\cdot))$ is the solution of the adjoint system \eqref{eq:adjj}.

Testing  \eqref{eq:adjj} with $S_{\eps}'(\ell_{\eps})(\dl)$ and \eqref{eq:ode_lin_q_e} with $(\xi_{\eps}, w_{\eps})$ yields 
\small
$$(w_{\eps},\dl)_{L^2(0,T;\lo)}=j'(S_{\eps}(\ell_{\eps})(S_{\eps}'(\ell_{\eps})(\dl)),$$ \normalsize which inserted in \eqref{nec} gives 
\begin{equation}\label{nec4}
(w_{\eps},\dl )_{L^2(0,T;\lo)}+(\ell_{\eps},\dl)_{H^1(0,T;\lo)}
+(\ell_{\eps}-\ll,\dl)_{H^1(0,T;\lo)}=0
\end{equation} for all $\dl \in H^1(0,T;\lo).$ 
Further, we observe that \begin{subequations} \begin{gather}
\partial_q j(S_{\eps}(\ell_{\eps})) \to \partial_q j(S(\ll)) \ \ \text{ in } L^2(0,T;L^{2}(\Omega)),\label{j_d}
\\\partial_\varphi j(S_{\eps}(\ell_{\eps}))  \to \partial_\varphi j(S(\ll))\ \ \text{ in } L^2(0,T;H^{1}(\Omega)^\ast), \label{j_phi}
\end{gather}\end{subequations}
in the light of \eqref{y_conv} combined with the continuous Fr\'echet-differentiability of $J$ (Assumption \ref{assu:j_ex}).
Next we focus on proving uniform bounds for the regularized adjoint states. By employing again a transformation of the variables where this time we abbreviate $\hat \xi_\eps:=\xi_\eps(T-\cdot)$ and by relying again on  the global Lipschitz-continuity of $\max_{\eps}$ and  \eqref{eq:h_adj}, we obtain from \eqref{eq:adj_syst1} 
\begin{align*}
\|\hat \xi_\eps (t)\|_{\lo}&\leq \int_0^t \| \beta \big(w_{\eps}(T-s,\hat \xi_\eps(s))-\frac{1}{\epsilon} \maxlim{_{\eps}}'(z_\eps(T-s))\hat \xi_\eps(s)\big)\|_{\lo}\;ds 
\\&\quad  +  \int_0^t \|[(f_\eps \circ \HH) '(q_\eps)]^\star \big(\frac{1}{\epsilon} \maxlim{_{\eps}}'(z_\eps)(\hat \xi_\eps(T-\cdot))\big)(T-s)\|_{\lo}\;ds 
\\&\qquad +\int_0^t\|\partial_q j(S_{\eps}(\ell_{\eps}))(T-s)\|_{\lo} \;ds 
\\&\leq \int_0^t c\,(\| \hat  \xi_\eps(s)\|_{\lo}+\|\partial_\varphi j(S_{\eps}(\ell_{\eps}))(T-s)\|_{\hoo})\;ds 
\\& \quad +  \int_0^t \hat L_f L_\HH \underbrace{\int_{T-s}^T \|\hat \xi_\eps(T-\zeta)\|_{\lo}\;d\zeta}_{=\int_{0}^s \|\hat \xi_\eps(\zeta)\|_{\lo}\;d\zeta} \,ds 
\\&\qquad +\int_0^t\|\partial_q j(S_{\eps}(\ell_{\eps}))(T-s)\|_{\lo} \;ds \quad \forall\,t \in [0,T].
\end{align*}
Now, Gronwall's inequality gives in turn 
$$\| \xi_\eps (t)\|_{\lo}\leq \widetilde c \int_0^{T-t} \|\partial_\varphi j(S_{\eps}(\ell_{\eps}))(T-s)\|_{\hoo} +
\|\partial_q j(S_{\eps}(\ell_{\eps}))(T-s)\|_{\lo} \;ds 
$$ for all $t \in [0,T]$. Thus, by relying on \eqref{j_d}-\eqref{j_phi} and by estimating again as above in \eqref{eq:adj_syst1}, this time without integrating, one obtains that there exists a constant, independent of $\eps$, such that $$\|\xi_{\eps}\|_{H^1(0,T;\lo)} \leq c.$$  As a consequence, $$\lambda_{\eps}:=\frac{1}{\epsilon} \maxlim{_{\eps}}'(z_{\eps})\xi_{\eps}$$ and $$\mu_\eps:= f_\eps '(\HH(q_\eps))\lambda_{\eps}$$ are uniformly bounded in $L^\infty(0,T;\lo)$ (recall that $\max_\eps$ and $f_\eps$ are globally Lipschitz continuous with constants independent of $\eps$). 
From \eqref{eq:adj_syst2} we can further deduce that there exists a constant $c>0$, independent of $\eps$, such that $\|w_{\eps}\|_{L^2(0,T;\ho)} \leq c$, where we use again \eqref{j_phi}. Therefore, we can extract weakly convergent subsequences (denoted by the same symbol) so that
\begin{equation}\label{nec2}\begin{split}
w_{\eps} \weakly w \quad \text{in }L^2(0,T;\ho), \quad \xi_{\eps} \weakly \xi \quad \text{in }H^1(0,T;\lo),
\\\  \lambda_{\eps} \weakly^* \lambda,\quad \mu_\eps \weakly^* \mu \quad \text{in }L^\infty(0,T;\lo) \quad\text{as }\eps \to 0.
\end{split}
\end{equation}
Owing to  \eqref{nec2}, \eqref{j_d}, \eqref{j_phi} and \eqref{l_conv}, we can pass to the limit  in \eqref{eq:adjj}-\eqref{nec4}. This results in 
\begin{subequations}\begin{gather}
-\dot{\xi}-\beta \big(w-\lambda\big)+\HH'(\bar q)^\star \mu =\partial_q j(\bar  q, \bar  \varphi) \ \text{ in }\llz,\quad \xi(T)=0,
\\-\alpha  \Delta w  + \beta \big(w-\lambda\big)=\partial_\varphi j(\bar  q, \bar  \varphi) \ \text{ in }L^2(0,T;\hoo),
\\(w,\dl )_{L^2(0,T;\lo)}+(\ll,\dl)_{H^1(0,T;\lo)}=0 \quad 
\forall \dl \in H^1(0,T;\lo),
\end{gather}
\end{subequations}
where for the passage to the limit in \eqref{eq:adj_syst1} we also relied on the continuity of the derivative of $\HH$ (see Assumption \ref{assu:stand}.\ref{it:stand1}) combined with \eqref{y_conv}.

Now, it remains to prove that \eqref{eq:lambda1}-\eqref{eq:mu1} is true. To this end, we show that, for each $\delta >0$, we have 
\begin{subequations}\label{eq:add}\begin{gather}
\lambda=\frac{1}{\epsilon}\maxlim '(\bar z) \xi \quad \text{a.e.\ in }M_\delta,\label{eq:lambda}
\\\mu=f'(\HH(\bar q))\lambda \quad \text{a.e.\ in }\widehat M_\delta,\label{eq:mu}
\end{gather}\end{subequations}where we abbreviate 
$M_\delta:=\{(t,x):|\bar z(t,x)| \geq \delta\}$, $\bar z:=-\beta(\bar q-\bar \varphi)-(f \circ \HH)(\bar q)$, and $\widehat M_\delta:=\{(t,x):|\HH(\bar q)(t,x)-n_f| \geq \delta\}$.

We begin by observing that 
\begin{align*}
 \|\HH(q_\eps)(t)-\HH(\bar q)(t) \|_{\li} \leq  \widehat L_{\HH} \|q_\eps-\bar q \|_{L^1(0,T;\li)} \quad  \text{ a.e.\ in }(0,T),
\end{align*}
in light of  Assumption \ref{assu:stand_eps}.\ref{it:s1}. Thus, as a consequence of \eqref{y_conv2}, we have  
\begin{equation}\label{hq}
\HH(q_\eps) \to \HH(\bar q) \quad \text{ in }L^\infty (0,T;\li),
\end{equation}
 which then implies $$z_\eps \to \bar z \quad \text{ in }\lii,$$ by  the Lipschitz continuity of  $f$ and  Assumption \ref{assu:f_eps}.\ref{it:feps1}. This means that $|z_\eps(t,x)| \geq \delta/2$ f.a.a.\,$(t,x) \in M_\delta$ for $\eps$ small enough, independent of $(t,x)$.
In view of the definition of $\maxlim{_{\eps}}$  we have 
$$\maxlim{_{\eps}}'(z_\eps(\cdot))=\maxlim'(\bar z(\cdot)) \quad \text{a.e.\ in }M_\delta$$for $\eps \leq \delta/2.$
The definition of $\lambda_\eps$ and \eqref{nec2} now yield \eqref{eq:lambda}.
To show \eqref{eq:mu}, we proceed in a similar way. Thanks to \eqref{hq}, there exists an $\eps$ small enough, independent of $(t,x)$, so that 
$|\HH(q_\eps)(t,x)-n_f| \geq \delta/2$ f.a.a.\,$(t,x) \in \widehat M_\delta$. 
Assumption \ref{assu:f_eps}.\ref{assu:f_epss} applied for $\delta/2$ then gives in turn the convergence
$$f_\eps'(\HH(q_\eps))- f'(\HH(q_\eps)) \to 0 \quad \text{ in }L^\infty (\widehat M_\delta).$$ 
As another consequence of Assumption \ref{assu:f_eps}.\ref{assu:f_epss}, we obtain that $f'$ is continuous on $(-\infty, n_f-\delta/2] \cup [n_f+\delta/2,\infty)$ since $f_\eps' \in C^1(\R),$ by assumption. Now, \eqref{hq}, Assumption \ref{assu:stand}.\ref{it:stand2} and Lebesgue dominated convergence imply that 
$$f'(\HH(q_\eps))- f'(\HH(\bar q)) \to 0 \quad \text{ in }L^2 (\widehat M_\delta).$$ 
Finally, the convergence of $\{\lambda_\eps\}$ from \eqref{nec2} along with the definition of $\mu_\eps$ yield that $$\mu_\eps \weakly \mu=f'(\HH(\bar q))\lambda \quad \text{ in }L^1 (\widehat M_\delta),$$i.e., \eqref{eq:mu}. Since $\delta >0$ was arbitrary and since $\underset{\delta>0}{\cup} M_\delta =\{(t,x):\bar z (t,x) \neq 0\}$ and $\underset{\delta>0}{\cup} \widehat M_\delta =\{(t,x):\HH(\bar q) (t,x) \neq n_f\}$ (up to a set of measure zero), the proof is now complete.
\end{proof}

\begin{remark}\label{rem:kkt}
\begin{itemize}
\item If $\bar z(t,x) \neq 0$ and if ${\HH(\bar q)(t,x) \neq n_f}$ a.e.\ in $(0,T) \times \O$, then the optimality system in 
Proposition \ref{lem:reg} coincides with the very same optimality conditions which one obtains when directly applying the KKT-theory {to} \eqref{eq:optt}, cf.\ \cite{troe}.
Moreover, we observe that \eqref{eq:adjjj} does not contain any information as to what happens in those $(t,x)$ for which $\bar z(t,x)$ and $\HH(\bar q)(t,x)$ are  non-smooth points of the mappings $\max$ and $f$, respectively. This is the focus of the next section, where the optimality conditions from Proposition \ref{lem:reg}  shall be  improved.
\item Indeed, \eqref{eq:adjjj} is not the best optimality system one could obtain via regularization. Such a system should also contain the relations
\begin{subequations}\label{conc}\begin{gather}
\lambda(t,x)\in \frac{1}{\epsilon} \partial \max(0) \xi(t,x) \quad \text{a.e.\ where  }\bar z(t,x) = 0,
\\ \mu(t,x) \in \partial f(n_f) \lambda(t,x) \quad \text{a.e.\ where  }\HH(\bar q)(t,x) =n_f.
\end{gather}\end{subequations}
We acknowledge the results \cite[Thm.\,2.4]{tiba}, \cite[Prop.\,2.17]{mcrf}, \cite[Thm.\,4.4]{cr_meyer}, where the respective limit optimality systems, though not strong stationary, include such relations between multipliers and adjoint states on the sets where the non-smoothness is active. We cannot expect this to happen in the present paper; by contrast to the aforementioned contributions, our adjoint state $\xi_\eps \in H^1(0,T;\lo)$ converges weakly in a space which is not compactly embedded in a Lebesgue space. Although we are able to show $$\maxlim{_\eps}'(\bar z_\eps(\cdot)) \weakly^* \gamma \in \partial \max(\bar z(\cdot)) \text{ in }L^\infty((0,T )\times\O),$$
this does not help us conclude \eqref{conc}, in view of the lack of space regularity of the adjoint state.
\end{itemize}
\end{remark}

\subsection{Towards strong stationarity}\label{sec:t}
In this section, we aim to derive a stronger optimality system than \eqref{eq:adjjj}. To this end, we will employ  arguments from previous works \cite{paper, st_coup}, which are entirely based on the limited differentiability properties of the non-smooth mappings involved. We begin by stating the  first order necessary optimality conditions in primal form.
\begin{lemma}[B-stationarity] If $\bar \ell \in \hs$ is locally optimal for \eqref{eq:optt}, then there holds
 \begin{equation}\label{eq:vi}
 j'(S(\bar \ell))S'(\bar \ell;\dl)+(\bar \ell , \dl )_{H^1(0,T;\lo)}   \geq 0   \quad \forall\, \dl\in \hs. \end{equation}
\end{lemma}\vspace{-0.4cm}
\begin{proof} As a result of Proposition \ref{lem:rd} and  Assumption \ref{assu:j_ex} we have  that  the composite mapping 
 $\hs \ni \ell \mapsto J(S(\ell),\ell) \in \R$ is (Hadamard) directionally differentiable \cite[Def.\ 3.1.1]{schirotzek} at $\ll$ in any direction $\dl$ with 
 directional derivative $  \partial_{(q,\varphi)} J(S(\bar \ell),\bar \ell )S'(\bar \ell;\dl)+\partial_{\ell} J(S(\bar \ell),\bar \ell ) \dl$; see  \cite[Lem.\ 3.1.2(b)]{schirotzek} and \cite[Prop.\ 3.6(i)]{shapiro}.
 The result then follows immediately from the local optimality of $\bar \ell$ and Assumption \ref{assu:j_ex}. 
\end{proof}

In order to improve the optimality conditions from the previous section \ref{sec:reg}, we make use of the following very natural requirement:

\begin{assumption}\label{assu:mon}
The \textit{history operator} $\HH$ satisfies the  monotonicity condition
$$\HH(q_1) \geq \HH(q_2) \quad \forall\,q_1,q_2 \in \llz \text{ with }q_1 \geq q_2.$$
\end{assumption}
\begin{remark}
It is self-evident that the cumulated damage  $\HH(q)$ (fatigue level of the material) increases as the damage $q$ increases. Hence, 
 the condition in Assumption \ref{assu:mon} is always satisfied in applications. \end{remark}
 
As an immediate consequence of Assumption \ref{assu:mon}, we have 
\begin{equation}\label{eq:d_h}
\HH'( q)(\eta)=\lim_{\tau \searrow 0} \frac{\HH( q+\tau \eta)-\HH( q)}{\tau} \geq 0\quad  \text{ a.e.\ in }(0,T) \times \O
\end{equation}
for all $q,\eta \in \llz$ with $\eta \geq 0\text{ a.e.\ in }(0,T) \times \O$.

The main result of this section reads as follows.
\begin{theorem}\label{thm:ss_qsep}
Suppose that Assumptions \ref{assu:j_ex},  \ref{assu:stand_eps},  \ref{assu:f_eps} and \ref{assu:mon} are fulfilled.
 Let $\bar \ell \in \hs$ be locally optimal for \eqref{eq:optt} with associated states  \begin{equation*}
  \bar q \in \hhyn \quad \text{and} \quad \bar \varphi \in \llu. 
 \end{equation*}   Then, there exist  adjoint states
 \begin{equation*}
  \xi \in H^{1}_T(0,T;\lo) \quad \text{and} \quad w \in \llu,
 \end{equation*}  
 and  multipliers $\lambda \in L^{\infty}(0,T;\lo)$ and $\mu \in L^\infty(0,T;\lo)$ such that the following system is satisfied 
 \begin{subequations}\label{eq:strongstat_q}
 \begin{gather}
-\dot{\xi}-\beta \big(w-\lambda\big)+\HH'(\bar q)^\star (\mu)=\partial_q j(\bar  q, \bar  \varphi) \ \text{ in }\llz,\quad \xi(T)=0,\label{eq:adjoint1_q}
\\-\alpha  \Delta w  + \beta \big(w-\lambda\big)=\partial_\varphi j(\bar  q, \bar  \varphi) \ \text{ in }L^2(0,T;\hoo),\label{eq:adjoint2_q}
 \\\left. \begin{aligned} 
  \lambda(t,x)&=\frac{1}{\epsilon} \raisebox{3pt}{$\chi$}_{\{ \bar  z>0\}}(t,x) \xi(t,x) \quad \text{a.e.\ where  }\bar z(t,x) \neq 0,\\
  \mu(t,x) &=f'(\HH(\bar q)(t,x)) \lambda(t,x) \quad \text{a.e.\ where  }\HH(\bar q)(t,x) \neq n_f,\end{aligned}\right\}\label{eq:kkt2}
 \\\left. \begin{aligned} 
 0 \leq \lambda(t,x)&\leq \frac{1}{\epsilon} (\xi(t,x) +G^+(t,x)) \quad \text{a.e.\ where  }\bar z(t,x) = 0,\\
  G^-(t,x) &\leq 0 \leq G^+(t,x) \quad \text{a.e.\ where  }\bar z(t,x) > 0,
  \end{aligned}\right\}\label{eq:signcond}
\\[1mm] (w,\dl )_{L^2(0,T;\lo)}+(\ll,\dl)_{H^1(0,T;\lo)}=0 \quad 
\forall \dl \in H^1(0,T;\lo),\label{eq:grad_q}
 \end{gather} 
 \end{subequations}    where we abbreviate $\bar z:=-\beta(\bar q-\bar \varphi)-(f \circ \HH)(\bar q)$. 
In \eqref{eq:signcond}, the mappings $G^+,G^- :[0,T] \times \O$ are defined as follows
 \begin{equation}\label{eq:g_plus}
  \begin{aligned}
  G^+(t,x)&:= \int_t^T  \HH'(\bar q)^\star[\raisebox{3pt}{$\chi$}_{\{\HH(\bar q)=n_f\}}(-\lambda f_+'(n_f)+\mu)](s,x)\;ds,
  \\G^-(t,x)&:= \int_t^T  \HH'(\bar q)^\star[\raisebox{3pt}{$\chi$}_{\{\HH(\bar q)=n_f\}}(-\lambda f_-'(n_f)+\mu)](s,x)\;ds,  \end{aligned} 
\end{equation}where, for any $v\in \R$, the right- and left-sided derivative of $f: \R \to \R$  are given by 
  $f'_+(v) := f'(v;1)$ and $ f'_-(v) := - f'(v;-1)$,  respectively. \end{theorem} 

\begin{proof}
The existence of a tuple $(\xi,w,\lambda,\mu)\in H^1(0,T;\lo) \times L^2(0,T;\ho) \times  L^\infty(0,T;\lo) \times L^\infty(0,T;\lo)$ satisfying the  system \eqref{eq:adjoint1_q}-\eqref{eq:adjoint2_q}-\eqref{eq:kkt2}-\eqref{eq:grad_q} is due to Proposition \ref{lem:reg}. Thus, the rest of the proof is focused on {showing \eqref{eq:signcond}}. In this context, we first follow the ideas from \cite[Proof of Lem.\ 2.8]{st_coup} and prove that the set of arguments of $\maxlim ' (\bar z;\cdot)$ from \eqref{eq:ode_lin_q1} is dense in $\llz$ (step (I)). {With this information at hand, we are then able to  show the desired result by employing a technique from  \cite[Proof of Thm.\ 2.11]{st_coup}, see also 
 \cite[Proof of Thm.\ 5.3]{paper} (step (II)).}

(I) Let $\rho \in \llz$ be arbitrary, but fixed. 
As indicated above, we next  show that there exists $\{\dl_n\} \subset H^1(0,T;\lo)$ such that 
 \begin{equation}\label{eq:rho}
  \begin{aligned}[t]
 \underbrace{ -\beta(\delta q_n - \delta \varphi_n)-(f \circ \HH)'(\bar q;\delta q_n)}_{:=\rho_n} \to \rho \quad \text{in }\llz \quad \text{as }n \to \infty,
  \end{aligned} 
\end{equation}where we abbreviate
 $(\delta q_n, \delta \varphi_n) := S'(\ll;\dl_n)$ and $\rho_n:= -\beta(\delta q_n - \delta \varphi_n)-(f \circ \HH)'(\bar q)(\delta q_n)$ for all $n \in \N$.  To this end, we follow the lines of the proof of  \cite[Lem.\ 2.8]{st_coup}. 
We start by noticing that the mapping 
  \begin{equation*}\label{eq:y_hat}
[0,T] \ni t \mapsto \hat q(t) \in \lo, \quad \hat q(t):=\frac{1}{\epsilon} \int_0^t  \maxlim ' (\bar z(s); \rho(s))\;ds  \end{equation*}
satisfies  $\hat q(0)=0$ and $\hat q \in H^1(0,T;\lo)$. Then, we  observe that $\hat q$ fulfills 
 \begin{equation}\label{eq:yh}
   \frac{d}{dt} \hat q(t) = \frac{1}{\epsilon} \maxlim '(\bar z(t);   -\beta \hat q(t) -(f \circ \HH)'(\bar q;\hat q)(t) +\rho(t) + \beta \hat q(t) +(f \circ \HH)'(\bar q;\hat q)(t)\big) \quad \text{\ae in } (0,T).\end{equation}
In view of the embedding $H^1(0,T;C_c^\infty(\O)) \dense \llz$, there exists a sequence $\{ \hat \varphi_n\}_n \subset H^1(0,T;C_c^\infty(\O))$ such that  
  \begin{equation}\label{eq:u_hat}
\beta \hat \varphi_n \to \rho +  \beta \hat q +(f \circ \HH)'(\bar q;\hat q) \quad \text{in } \llz \ \ \text{as } n \to \infty.
 \end{equation}
For any $n \in \N$, consider the equation
 \begin{equation}\label{eq:ym}
   \frac{d}{dt} \hat q_n(t) =\frac{1}{\epsilon}  \maxlim '(\bar z(t);  -\beta(\hat q_n - \hat \varphi_n)-(f \circ \HH)'(\bar q;\hat q_n) \big) \quad \text{\ae in } (0,T),\ \hat q_n(0) = 0.\\ \end{equation}
By arguing as in the proof of Lemma \ref{lem:rd} we see  that \eqref{eq:ym} admits a unique solution $\hat q_n \in H^1_0(0,T;\lo)$. Now, we define
   \begin{equation}\label{eq:l_hat}
 \delta \ell_n:=   -\alpha  \Delta \hat \varphi_n  + \beta \big(\hat \varphi_n-\hat q_n \big) \in \hs,
 \end{equation} 
 such that the pair $(\hat q_n,\hat \varphi_n) $ solves the system \eqref{eq:ode_lin_q} associated to $\ll$ with right-hand side $\delta \ell_n \in \hs$; note that the regularity of $\dl_n$ in \eqref{eq:l_hat} is due to the $H^1(0,T;C_c^\infty(\O))$-regularity of $\hat \varphi_n$.
 In view of the unique solvability of \eqref{eq:ode_lin_q}, cf.\ Proposition \ref{lem:rd}, $(\hat q_n,\hat \varphi_n)= S'(\ll;\dl_n)$.
Owing to the Lipschitz-continuity of the directional derivative of $\max$ (w.r.t.\ direction) and \eqref{eq:h'} , we further  obtain from \eqref{eq:yh} and \eqref{eq:ym} 
\begin{align*}
\epsilon\| (\hat q_n-\hat q)(t)\|_{\lo} &\leq \beta\,  \int_0^t \|(\hat q-\hat q_n)(s) \|_{\lo} \,ds+L_f    \, L_{\HH}\int_0^t \int_0^s \|(\hat q-\hat q_n)(\zeta) \|_{\lo}  \,d\zeta \;ds
\\&\qquad + \int_0^t \|-\beta \hat \varphi_n (s)+\rho(s)+  \beta \hat q(s) +(f \circ \HH)'(\bar q;\hat q)(s)\|_{\lo} \,ds.
\end{align*}Gronwall's inequality and \eqref{eq:u_hat} then give in turn 
  \begin{equation}\label{eq:conv}
\| \hat q_n-\hat q\|_{H^1(0,T;\lo)} \leq c\,\|-\beta \hat \varphi_n +\rho+  \beta \hat q +(f \circ \HH)'(\bar q;\hat q)\|_{\llz}\to 0 \ \ \text{as } n \to \infty ,\end{equation}where $c>0$ is a constant dependent only on the given data. 
By relying on the continuity of $\F'(\bar q;\cdot):\llz \to \llz$, cf.\,\eqref{eq:h'}, we have
  \begin{equation}\label{eq:uh}
\beta \hat q_n +(f \circ \HH)'(\bar q;\hat q_n) \to \beta \hat q +(f \circ \HH)'(\bar q;\hat q) \quad \text{in } \llz \ \ \text{as } n \to \infty,
 \end{equation}as a result of \eqref{eq:conv}.
Combining \eqref{eq:u_hat} and \eqref{eq:uh} finally yields 
  \begin{equation*}
 -\beta(\hat q_n - \hat \varphi_n)-(f \circ \HH)'(\bar q;\hat q_n) \to \rho \quad \text{in }\llz \quad \text{as }n \to \infty.
 \end{equation*}Since we established above that $(\hat q_n,\hat \varphi_n)= S'(\ll;\dl_n)$, the proof of this step is now complete. 

(II) 
In the following,  $\rho \in \llz$ remains arbitrary, but fixed. 
To prove the desired relations in  \eqref{eq:signcond}, we first make use of the B-stationarity from\  \eqref{eq:vi}. 
Here we test with the function $\dl_n \in \hs$ which was defined  in \eqref{eq:l_hat}. 

We test \eqref{eq:adjoint1_q},  \eqref{eq:adjoint2_q}, and \eqref{eq:grad_q} with $(\delta q_n, \delta \varphi_n) := S'(\ll;\dl_n)$  and $\dl_n$, respectively. 
This leads to 
  \begin{equation}\label{eq:for_ss}
  \begin{aligned}
 & 0 \leq \partial_q j(\bar  q, \bar  \varphi) \dd_n+\partial_{\varphi} j(\bar  q, \bar  \varphi) \df_n+ (\ll, \dl_n)_{H^1(0,T;\lo)} 
 \\&\qquad =-\int_0^T (\dot \xi(t),\dd_n(t))_{\lo} \;dt -\beta (w- \lambda, \dd_n)_{\llz}
 \\&\qquad \quad  +(\HH'(\bar q)^\star (\mu ),\dd_n)_{\llz}+\beta (w- \lambda, \delta \varphi_n)_{\llz} 
 \\&\qquad \quad  +\alpha ( \nabla w  ,\nabla \delta \varphi_n)_{\llz}
 - (w, \dl_n)_{\llz}
\\&\qquad =\int_0^T (\xi(t),\dot \dd_n(t))_{\lo} \;dt -  ( \lambda,-\beta(\dd_n-  \df_n) )_{\llz}
\\&\qquad \quad +(\mu,\HH'(\bar q)(\dd_n))_{\llz} -\dual{ \underbrace{\beta ( \dd_n- \delta \varphi_n)+\alpha  \Delta \df_n +\dl_n}_{=0,\text{ cf.}\ \eqref{eq:ode_lin_q2}}}{w}_{L^2(0,T;\ho)}
\\&\qquad \underbrace{=}_{\eqref{eq:ode_lin_q1}} \int_0^T (\xi(t),\frac{1}{\epsilon}  \maxlim ' (\bar z(t);\rho_n(t))_{\lo} \;dt - \int_0^T ( \lambda(t),\rho_n(t))_{\lo} \;dt
\\&\qquad-\int_0^T (\lambda(t), \F'(\bar q;\dd_n)(t))_{\lo}\,dt+ \int_0^T (\mu(t),\HH'(\bar q)(\dd_n)(t))_{\lo}\,dt
 \quad  \forall\, n \in \N,
  \end{aligned} 
\end{equation}
where the second identity follows from integration by parts, $\dd_n(0)=0$, and $\xi(T)=0$; here we also recall the abbreviation $\rho_n:= -\beta(\delta q_n - \delta \varphi_n)-(f \circ \HH)'(\bar q;\delta q_n)$, see \eqref{eq:rho}.  
In view of \eqref{eq:rho} and since $\dd_n(t)=\frac{1}{\epsilon} \int_0^t   \maxlim ' (\bar z(s);\rho_n(s))\,ds$, letting $n \to \infty$ in \eqref{eq:for_ss} leads to 
 \begin{equation}\label{eq:for_sss}\begin{aligned}
  0 &\leq \int_0^T (\xi(t),\frac{1}{\epsilon}  \maxlim ' (\bar z(t);\rho(t))_{\lo} \;dt - \int_0^T ( \lambda(t),\rho(t))_{\lo} \;dt
\\ \quad & - \int_0^T (\lambda(t), \F'(\bar q;\widehat q_\rho)(t))_{\lo}\,dt+ \int_0^T (\mu(t),\HH'(\bar q)(\widehat q_\rho)(t))_{\lo}\,dt
   \end{aligned}
\end{equation}for all $\rho \in \llz$, where we abbreviate
 \begin{equation}\label{eq:q_rho}
 \widehat q_\rho(t):=\frac{1}{\epsilon} \int_0^t  \maxlim ' (\bar z(s); \rho(s))\;ds \quad \forall\,t \in [0,T].
 \end{equation} 
 Here we used the fact that $\maxlim '(\bar z; \cdot):\llz \to \llz$ is continuous, by the Lipschitz-continuity of $\max$, as well as \eqref{eq:conv} in combination with \eqref{eq:h'} and the fact that $\HH'(\bar q) \in \LL(\llz,\llz)$. 

Next, we take a closer look at the second line in the estimate \eqref{eq:for_sss}. In this context, we first notice  that, for all $v,h \in \R$, it holds 
\begin{equation}\label{eq:d_f}
f'(v;h)=\begin{cases}f'_+(v)h,\quad \text{if }h \geq 0,\\f'_-(v)h,\quad \text{if }h < 0.\end{cases}
\end{equation}
Moreover, we recall that
 \begin{equation}\label{eq:max}
  \maxlim{'}(v;h)
  =\begin{cases}
  h & \text{if } v>0, \\
  \max\{h,0\} &\text{if } v=0,\\
  0 &\text{if } v<0.
  \end{cases}
 \end{equation}
Now, let $\rho \in \llz$ with $\rho \geq 0 \text{ a.e.\ in }(0,T) \times \O$ be arbitrary, but fixed. In view of  \eqref{eq:q_rho} and \eqref{eq:max}, we have $\widehat q_\rho \geq 0 \text{ a.e.\ in }(0,T) \times \O$ and \eqref{eq:d_h} implies
\begin{equation*} \label{for}
\HH'(\bar q)(\widehat q_\rho) \geq 0\quad  \text{ a.e.\ in }(0,T) \times \O.
\end{equation*} Then, by recalling \eqref{eq:f_h_deriv} and by employing Fubini's theorem, we obtain 
 \begin{equation}\label{eq:2t1}\begin{aligned}
- \int_0^T & (\lambda(t), \F'(\bar q;\widehat q_\rho)(t))_{\lo}\,dt+ \int_0^T (\mu(t),\HH'(\bar q)(\widehat q_\rho)(t))_{\lo}\,dt\\&=
\int_0^T \int_\O [-\lambda(t,x) f_+'(\HH(\bar q)(t,x))+\mu(t,x)]\HH'(\bar q)(\widehat q_\rho)(t,x)\,dx\,dt
\\&=
\int_0^T \int_\O \HH'(\bar q)^\star[-\lambda f_+'(\HH(\bar q))+\mu](t,x)\widehat q_\rho(t,x)\,dx\,dt
\\&\underset{\eqref{eq:q_rho}}{=}
 \int_\O \int_0^T \HH'(\bar q)^\star[-\lambda f_+'(\HH(\bar q))+\mu](t,x)\Big(\frac{1}{\epsilon} \int_0^t  \maxlim ' (\bar z(s,x); \rho(s,x))\;ds \Big) \,dt\,dx
 \\&=
 \int_\O \int_0^T \frac{1}{\epsilon}   \maxlim ' (\bar z(t,x); \rho(t,x)) {\Big(  \int_t^T  \HH'(\bar q)^\star[-\lambda f_+'(\HH(\bar q))+\mu](s,x)\;ds\Big)} \,dt\,dx
  \\&=
\int_0^T \int_\O  \frac{1}{\epsilon}   \maxlim ' (\bar z(t,x); \rho(t,x)) {G^+(t,x)} \,dx\,dt
   \quad  \forall\, \rho \in \llz, \rho \geq 0,
   \end{aligned}
\end{equation}where the last equality is due to the definition of $G^+$ in \eqref{eq:g_plus} combined with the second identity in \eqref{eq:kkt2}. 
Going back to \eqref{eq:for_sss}, we have 
\begin{equation}\label{eq:for_sss1}\begin{aligned}
  0 &\leq \int_0^T \int_\O \frac{1}{\epsilon}  \maxlim ' (\bar z(t,x);\rho(t,x))\xi(t,x)-\lambda(t,x)\rho(t,x)\,dx \;dt
\\ \quad & + \int_0^T \int_\O  \frac{1}{\epsilon}   \maxlim ' (\bar z(t,x); \rho(t,x)) G^+(t,x) \,dx\,dt   \quad  \forall\, \rho \in \llz, \rho \geq 0.
   \end{aligned}
\end{equation}By means of the fundamental lemma of calculus of variations in combination with the positive homogeneity of the directional derivative w.r.t.\ direction, we deduce from \eqref{eq:for_sss1} the inequality
\begin{equation}\label{eq:ae1}\begin{aligned}
 \frac{1}{\epsilon}  \maxlim ' (\bar z(t,x);1)\xi(t,x)-\lambda(t,x)+  \frac{1}{\epsilon}   \maxlim ' (\bar z(t,x); 1) G^+(t,x)  \geq 0 \quad \text{a.e.\ in } (0,T) \times \O.
   \end{aligned}
\end{equation}
By arguing exactly in the same way as above, where {one takes into account the fact that $\HH'(\bar q)(\widehat q_\rho) \leq 0 \text{ a.e.\ in }(0,T) \times \O$, for $\rho \leq 0 \text{ a.e.\ in }(0,T) \times \O$,} we show
 \begin{equation}\label{eq:2t}\begin{aligned}
- \int_0^T & (\lambda(t), \F'(\bar q;\widehat q_\rho)(t))_{\lo}\,dt+ \int_0^T (\mu(t),\HH'(\bar q)(\widehat q_\rho)(t))_{\lo}\,dt\\&
=
 \int_0^T  \int_\O  \frac{1}{\epsilon}   \maxlim ' (\bar z(t,x); \rho(t,x)) \underbrace{\Big(  \int_t^T  \HH'(\bar q)^\star[-\lambda f_-'(\HH(\bar q))+\mu](s,x)\;ds\Big)}_{=G^-(t,x)} \,dx \,dt
 \\& \qquad \qquad \qquad   \qquad \qquad \qquad  \qquad\qquad \qquad \qquad    \qquad  \forall\, \rho \in \llz, \rho \leq 0.
   \end{aligned}
\end{equation}This gives in turn 
\begin{equation}\label{eq:ae2}\begin{aligned}
 \frac{1}{\epsilon}  \maxlim ' (\bar z(t,x);-1)\xi(t,x)+\lambda(t,x)+  \frac{1}{\epsilon}   \maxlim ' (\bar z(t,x); -1) G^-(t,x)  \geq 0 \quad \text{a.e.\ in } (0,T) \times \O,
   \end{aligned}
\end{equation}where we relied again on  the fundamental lemma of calculus of variations and  the positive homogeneity of the directional derivative w.r.t.\ direction. From \eqref{eq:ae1}- \eqref{eq:ae2} and the fact that $ \maxlim ' (0;\cdot)=\max\{\cdot, 0\}$ (see \eqref{eq:max}) we can now conclude the first relation in \eqref{eq:signcond}. Finally, the second relation in \eqref{eq:signcond} is a consequence of \eqref{eq:kkt2}, \eqref{eq:ae1}- \eqref{eq:ae2} and \eqref{eq:max}. This completes the proof.
\end{proof}

\begin{corollary}[Strong stationarity in the case that $f$ is smooth]\label{cor}
Suppose that Assumptions \ref{assu:j_ex} and \ref{assu:stand_eps} are fulfilled.
Let $\bar \ell \in \hs$ be locally optimal for \eqref{eq:optt} with associated states  \begin{equation*}
  \bar q \in \hhyn \quad \text{and} \quad \bar \varphi \in \llu. 
 \end{equation*}   If the set $\{(t,x) \in (0,T) \times \O: \HH(\bar q)(t,x)=n_f\}$ has measure zero, then  there exist unique adjoint states
 \begin{equation*}
  \xi \in H^{1}_T(0,T;\lo) \quad \text{and} \quad w \in \llu,
 \end{equation*}  
 and a unique multiplier $\lambda \in L^{\infty}(0,T;\lo)$ such that the following system is satisfied 
 \vspace{-0.5cm}
 \begin{subequations}\label{eq:strongstat}
 \begin{gather}
-\dot{\xi}-\beta \big(w-\lambda\big)+[\F'(\bar q)]^\star (\lambda )=\partial_q j(\bar  q, \bar  \varphi) \ \text{ in }\llz, \quad \xi(T) = 0,\label{eq:adjoint1} \\[1mm]
 -\alpha  \Delta w  + \beta \big(w-\lambda\big)=\partial_\varphi j(\bar  q, \bar  \varphi) \ \ \text{ in } L^2(0,T;H^{1}(\Omega)^\ast), \label{eq:adjoint2}\\[1mm]
 \left. \begin{aligned} 
 & \lambda(t,x)=\frac{1}{\epsilon} \raisebox{3pt}{$\chi$}_{\{ \bar  z>0\}}(t,x) \xi(t,x) \quad \text{a.e.\ where  }\bar z(t,x) \neq 0,\\
  &0 \leq \lambda(t,x) \leq \frac{1}{\epsilon} \xi(t,x)  \quad \text{a.e.\ where  }\bar z(t,x) = 0,\end{aligned}\right\}\label{eq:signcond_q}
\\[1mm] (w,\dl )_{L^2(0,T;\lo)}+(\ll,\dl)_{H^1(0,T;\lo)}=0 \quad 
\forall \dl \in H^1(0,T;\lo),\label{eq:grad}
 \end{gather} 
 \end{subequations}    where we abbreviate $\bar z:=-\beta(\bar q-\bar \varphi)-(f \circ \HH)(\bar q)$. Moreover, \eqref{eq:strongstat} is of strong stationary type, i.e.,  if $\bar \ell \in H^1(0,T;L^2(\O))$ together with its states $(\bar q, \bar \varphi) \in \hhyn \times \llu$, some 
 adjoint states $(\xi,w) \in H^{1}_T(0,T;\lo) \times \llu$, and a multiplier $\lambda \in L^{\infty}(0,T;\lo)$ 
 satisfy the optimality system \eqref{eq:adjoint1}--\eqref{eq:grad}, then it also satisfies the variational inequality  \eqref{eq:vi}.
 \end{corollary}
 \begin{proof}
 The first statement is a consequence of Theorem \ref{thm:ss_qsep}. Note that here we do not ask that Assumption   \ref{assu:f_eps} holds true; $f$ does not need to be smoothened, as its non-smoothness is never active. Assumption   \ref{assu:mon} is also not required here; this was necessary in the proof of Theorem \ref{thm:ss_qsep} only to show \eqref{eq:2t1} and \eqref{eq:2t}. Since $\{(t,x) \in (0,T) \times \O: \HH(\bar q)(t,x)=n_f\}$ has measure zero, \eqref{eq:2t1} and \eqref{eq:2t} follow immediately from the second relation in \eqref{eq:kkt2}.

  To prove the second assertion, we  let $\rho \in \llz$ be arbitrary, but fixed and abbreviate $\rho^+:=\max\{\rho,0\}$ and $\rho^-:=\min\{\rho,0\}$.
By distinguishing between the sets $\{(t,x) \in (0,T) \times  \O: \bar z(t,x) >0 \}$, $\{(t,x) \in (0,T) \times  \O: \bar z(t,x) =0 \}$ and $\{(t,x) \in (0,T) \times  \O: \bar z(t,x) <0 \}$,  we obtain from \eqref{eq:signcond_q} and \eqref{eq:max}
\begin{equation}\label{eq:add_rho}\begin{aligned}
  0 &\quad \leq \int_0^T \int_\O \frac{1}{\epsilon}  [\maxlim ' (\bar z(t,x);\rho^+(t,x))+\maxlim ' (\bar z(t,x);\rho^-(t,x))]\xi(t,x) \,dx \;dt
  \\&\qquad- \int_0^T \int_\O \lambda(t,x)[\rho^+(t,x)+\rho^-(t,x)]\,dx \;dt
   \\&{=} \int_0^T \int_\O \frac{1}{\epsilon}  \maxlim ' (\bar z(t,x);\rho(t,x))\xi(t,x) \,dx \;dt
- \int_0^T \int_\O \lambda(t,x)\rho(t,x)\,dx \;dt 
  \end{aligned}
\end{equation} 
for all $\rho \in \llz$. Now, let $\dl \in \hs$ be arbitrary but fixed and 
 test \eqref{eq:adjoint1},  \eqref{eq:adjoint2}, and \eqref{eq:grad} with $(\delta q, \delta \varphi) := S'(\ll;\dl)$  and $\dl$, respectively. 
This leads to 
  \begin{equation*}
  \begin{aligned}
& \partial_q j(\bar  q, \bar  \varphi) \dd+\partial_{\varphi} j(\bar  q, \bar  \varphi) \df+ (\ll, \dl)_{H^1(0,T;\lo)} 
 \\&=-\int_0^T (\dot \xi(t),\dd(t))_{\lo} \;dt -\beta (w- \lambda, \dd)_{\llz}+([\F'(\bar q)]^\star (\lambda ),\dd)_{\llz}
 \\&\quad   +\beta (w- \lambda, \delta \varphi)_{\llz} +\alpha ( \nabla w  ,\nabla \delta \varphi)_{\llz}
 - (w, \dl)_{\llz}
\\&=\int_0^T (\xi(t),\dot \dd(t))_{\lo} \;dt -  ( \lambda,-\beta(\dd-  \df) )_{\llz}+(\lambda,\F'(\bar q)(\dd))_{\llz}
\\& -\dual{ \underbrace{\beta ( \dd- \delta \varphi)+\alpha  \Delta \df +\dl}_{=0,\text{ cf.}\ \eqref{eq:ode_lin_q2}}}{w}_{L^2(0,T;\ho)}
\\& \underbrace{=}_{\eqref{eq:ode_lin_q1}} \int_0^T (\xi(t),\frac{1}{\epsilon}  \maxlim ' (\bar z(t);(-\beta(\delta q - \delta \varphi)-(f \circ \HH)'(\bar q;\delta q))(t))_{\lo} \;dt 
\\&\qquad \qquad - \int_0^T ( \lambda(t),(-\beta(\delta q - \delta \varphi)-(f \circ \HH)'(\bar q;\delta q))(t))_{\lo} \;dt
\\& \underset{\eqref{eq:add_rho}}{\geq} 0,
  \end{aligned} 
\end{equation*}
where the second identity follows from integration by parts, $\dd(0)=0$, and $\xi(T)=0$. Since  
 $\dl \in \hs$ was arbitrary, the proof is now complete.
 \end{proof}
 
{\begin{remark} 
 We remark that if fatigue is not taken into consideration, i.e., if $f$ is replaced by a nonnegative constant, then \eqref{eq:strongstat} reduces to the strong stationary optimality conditions  obtained in \cite[Thm.\,4.5]{st_coup}; note that therein the control space is $\llz$ instead of $\hs.$
 \end{remark}}
\begin{remark}
As opposed to \eqref{eq:strongstat}, the optimality system in Theorem \ref{thm:ss_qsep} is not strong stationary, as we will see in the next section. However, we emphasize that \eqref{eq:strongstat_q} is a comparatively strong optimality system. While countless non-smooth problems have been addressed by resorting to a smoothening procedure as the one in the proof of Proposition \ref{lem:reg} (see e.g.\, \cite{ Barbu:1981:NCD,  Friedman1987, He1987} and the references therein), we went a step further and improved the optimality conditions from Proposition \ref{lem:reg} by proving  the additional information contained in \eqref{eq:signcond}. Let us point out  that  sign conditions on the sets where the non-smoothness is active, in our case
 \begin{align*} 
0 \leq \lambda(t,x)  \quad \text{a.e.\ where  }\bar z(t,x) = 0
  \end{align*}
are not expected to be obtained by classical regularization techniques, see e.g. \cite[Remark 3.9]{mcrf}.

\end{remark}

\subsection{Discussion of the  optimality system \eqref{eq:strongstat_q}. Comparison to strong stationarity}\label{sec:comp}

We begin this section by writing down how the strong stationary optimality conditions for the control of \eqref{eq:optt} should look like.

\begin{proposition}[An optimality system that implies B-stationarity]\label{prop:ss}
Suppose that Assumptions \ref{assu:j_ex} is  fulfilled.
Assume that $\bar \ell \in \hs$ together with its states $(\bar q, \bar \varphi) \in \hhyn \times \llu$, some 
 adjoint states $(\xi,w) \in H^{1}_T(0,T;\lo) \times \llu$, and some  multipliers $\lambda, \mu \in L^{\infty}(0,T;\lo)$ 
 satisfy the optimality system
 \begin{subequations}\label{eq:strongstat2}
 \begin{gather}
-\dot{\xi}-\beta \big(w-\lambda\big)+\HH'(\bar q)^\star (\mu)=\partial_q j(\bar  q, \bar  \varphi) \ \text{ in }\llz,\quad \xi(T)=0,\label{eq:adjoint1_q2}
\\-\alpha  \Delta w  + \beta \big(w-\lambda\big)=\partial_\varphi j(\bar  q, \bar  \varphi) \ \text{ in }L^2(0,T;\hoo),\label{eq:adjoint2_q2}
 \\\left. \begin{aligned} 
  \lambda(t,x)&=\frac{1}{\epsilon} \raisebox{3pt}{$\chi$}_{\{ \bar  z>0\}}(t,x) \xi(t,x) \quad \text{a.e.\ where  }\bar z(t,x) \neq 0,\\
  \mu(t,x) &=f'(\HH(\bar q)(t,x)) \lambda(t,x) \quad \text{a.e.\ where  }\HH(\bar q)(t,x) \neq n_f,\end{aligned}\right\}\label{eq:kkt22}
 \\\left. \begin{aligned} 
 0 \leq \lambda(t,x)&\leq \frac{1}{\epsilon} \xi(t,x)  \quad \text{a.e.\ where  }\bar z(t,x) = 0,\\
   f_+'(n_f) \lambda(t,x)  \leq \mu(t,x)&\leq  f_-'(n_f) \lambda(t,x)  \quad \text{a.e.\ where  }\HH(\bar q)(t,x) = n_f,
  \end{aligned}\right\}\label{eq:signcond2}
\\[1mm] (w,\dl )_{L^2(0,T;\lo)}+(\ll,\dl)_{H^1(0,T;\lo)}=0 \quad 
\forall \dl \in H^1(0,T;\lo),\label{eq:grad_q2}
 \end{gather} 
 \end{subequations}    where we abbreviate $\bar z:=-\beta(\bar q-\bar \varphi)-(f \circ \HH)(\bar q)$ and where, for any $v\in \R$, the right- and left-sided derivative of $f: \R \to \R$  are given by 
  $f'_+(v) := f'(v;1)$ and $ f'_-(v) := - f'(v;-1)$,  respectively. Then, $\ll$ also satisfies the variational inequality \eqref{eq:vi}.\end{proposition} 
  \begin{proof}
Let $\rho \in \llz$ be arbitrary, but fixed. In the proof of Corollary \ref{cor} we saw that the first identity in \eqref{eq:kkt22} and the first relation in \eqref{eq:signcond2} combined with \eqref{eq:max} imply 
\begin{equation}\label{eq:add_rho1}\begin{aligned}
  0 \leq  \int_0^T \int_\O \frac{1}{\epsilon}  \maxlim ' (\bar z(t,x);\rho(t,x))\xi(t,x) \,dx \;dt
- \int_0^T \int_\O \lambda(t,x)\rho(t,x)\,dx \;dt
  \end{aligned}
\end{equation}for all $\rho \in \llz.$
Next we abbreviate $\HH'(\bar q)(\widehat q_{\rho})^-:=\min\{\HH'(\bar q)(\widehat q_{\rho}),0\}$ and $\HH'(\bar q)(\widehat q_{\rho})^+:=\max\{\HH'(\bar q)(\widehat q_{\rho}),0\}$, where 
 \begin{equation*}
 \widehat q_\rho(t):=\frac{1}{\epsilon} \int_0^t  \maxlim ' (\bar z(s); \rho(s))\;ds \quad \forall\,t \in [0,T].
 \end{equation*}
 From the second identity in \eqref{eq:kkt22} and the second relation in \eqref{eq:signcond2} we deduce  that 
\begin{equation}\label{eq:add_rho2}\begin{aligned}
  0 & \leq \int_0^T \int_\O \underbrace{[-\lambda(t,x) f_+'(\HH(\bar q)(t,x))+\mu(t,x)]}_{\geq 0}\HH'(\bar q)(\widehat q_\rho)^+(t,x)\,dx\,dt
  \\&\qquad +\int_0^T \int_\O \underbrace{[-\lambda(t,x) f_-'(\HH(\bar q)(t,x))+\mu(t,x)]}_{\leq 0}\HH'(\bar q)(\widehat q_\rho)^-(t,x)\,dx\,dt
 \\ &=\int_0^T \int_\O -\lambda(t,x) f'(\HH(\bar q)(t,x);\HH'(\bar q)(\widehat q_\rho)^+(t,x))+\mu(t,x)\HH'(\bar q)(\widehat q_\rho)^+(t,x)  \,dx\,dt
  \\&\qquad +\int_0^T \int_\O-\lambda(t,x) f'(\HH(\bar q)(t,x);\HH'(\bar q)(\widehat q_\rho)^-(t,x))+\mu(t,x)\HH'(\bar q)(\widehat q_\rho)^-(t,x)\,dx\,dt
 \\&= - \int_0^T  (\lambda(t), \F'(\bar q;\widehat q_\rho)(t))_{\lo}\,dt+ \int_0^T (\mu(t),\HH'(\bar q)(\widehat q_\rho)(t))_{\lo}\,dt,
  \end{aligned}\end{equation}
where in the second identity  we relied on \eqref{eq:d_f}.
Adding \eqref{eq:add_rho1} and \eqref{eq:add_rho2} yields \eqref{eq:for_sss}. Now, let $\dl \in \hs$ be arbitrary but fixed and abbreviate $(\delta q, \delta \varphi) := S'(\ll;\dl)$. By  testing \eqref{eq:for_sss} with $-\beta(\delta q-\delta \varphi)-\F'(\bar q;\delta q)$ and by arguing step by step backwards as in the proof of \eqref{eq:for_ss}, we finally arrive at the desired result. 
\end{proof}

\begin{remark}
Some words concerning  Proposition \ref{prop:ss} are in order:
\begin{itemize}
\item The optimality system \eqref{eq:strongstat2} differs from \eqref{eq:strongstat_q} only regarding the relations in \eqref{eq:signcond2} and \eqref{eq:signcond}. 
As expected, the optimality conditions in \eqref{eq:signcond2} contain more information than \eqref{eq:signcond}. This is also confirmed by Proposition \ref{prop:ss2} below.
\item We point out that \eqref{eq:strongstat2} is not of strong stationary type, as we were not able to show  $\eqref{eq:vi}\Rightarrow \eqref{eq:strongstat2}$; the optimality conditions in \eqref{eq:strongstat2} just point out the information that   is missing in \eqref{eq:strongstat_q}, namely 
\begin{equation}\label{mis} \begin{aligned} 
\lambda(t,x)&\leq \frac{1}{\epsilon} \xi(t,x)  \quad \text{a.e.\ where  }\bar z(t,x) = 0,\\
   f_+'(n_f) \lambda(t,x)  \leq \mu(t,x)&\leq  f_-'(n_f) \lambda(t,x)  \quad \text{a.e.\ where  }\HH(\bar q)(t,x) = n_f.
  \end{aligned}\end{equation}
Note that the sign condition 
 \begin{align*} 
0 \leq \lambda(t,x)  \quad \text{a.e.\ where  }\bar z(t,x) = 0
  \end{align*} is already contained in \eqref{eq:signcond}. The proof of Proposition \ref{prop:ss} shows that \eqref{mis} is indeed needed for the  implication $\eqref{eq:strongstat_q} \Rightarrow \eqref{eq:vi}$.
  \item  In order to prove that a certain optimality system implies B-stationarity, it is essential that it includes sign conditions for the involved multipliers and/or adjoint states  on the sets where the non-smoothness is active. This fact has been observed in many contributuions dealing with strong stationarity \cite[Rem.\ 6.9]{paper}, \cite[Rem.\ 3.9]{mcrf}, \cite[Rem.\ 4.8]{st_coup}, \cite[Rem.\ 4.15]{cr_meyer}. In our case, see \eqref{eq:signcond}, the information on $\{\bar z=0\}$ is incomplete, while the sign conditions on the set $\{\HH(\bar q) = n_f\}$ are  non-existent and seem to be hidden in the integral formulations \eqref{eq:g_plus}.
\end{itemize}
\end{remark}

\begin{proposition}[The optimality system \eqref{eq:strongstat2} is stronger than \eqref{eq:strongstat_q}]\label{prop:ss2}
Suppose that all the hypotheses in Proposition \ref{prop:ss} are  fulfilled. If, in addition, Assumption \ref{assu:mon} holds true, then \eqref{eq:strongstat_q} is satisfied.
\end{proposition} 
  \begin{proof}We only need to show that \eqref{eq:signcond2} implies \eqref{eq:signcond}. To this end, we first prove that 
\begin{equation}\label{eq:h2}
\HH'(\bar q)^\star(\eta_1) \geq \HH'(\bar q)^\star(\eta_2) \quad \forall\,\eta_1,\eta_2 \in \llz \text{ with }\eta_1 \geq \eta_2.
\end{equation}
We recall that, as a consequence of  Assumption \ref{assu:mon}, $\HH'(\bar q)(\rho) \geq 0$ for all $\rho \in \llz, \rho \geq 0$, cf.\,\eqref{eq:d_h}. This leads to
\begin{align*}
(\HH'(\bar q)^\star(\eta_1),\rho)_{\llz}&=(\eta_1,\HH'(\bar q)(\rho))_{\llz}
\\&\quad  \geq (\eta_2,\HH'(\bar q)(\rho))_{\llz}
=(\HH'(\bar q)^\star(\eta_2),\rho)_{\llz},
\end{align*}from which \eqref{eq:h2} follows.
Now, the second relation in \eqref{eq:signcond2} and the definitions of $G^+$ and $G^-$ in \eqref{eq:g_plus} give in turn
$$G^+ \geq 0 \text{ and } G^- \leq 0 \quad \text{a.e.\ in }(0,T) \times \O.$$
Thus, \eqref{eq:signcond2} implies \eqref{eq:signcond} and the proof is complete.
\end{proof}

\begin{remark} \label{exp}
The gap between \eqref{eq:strongstat_q} and the  strong stationary  optimality conditions \eqref{eq:strongstat2} is  due to the additional non-smooth mapping $f$ appearing in the argument of the initial non-smoothness $\max$, cf.\,\eqref{eq:ode}. To see this, let us take a closer look at the proof of Theorem \ref{thm:ss_qsep}. Therein, \eqref{eq:signcond} is proven by relying on direct methods from previous works \cite{st_coup, paper} which deal with strong stationarity in the context of one non-differentiable map. In these findings it has been observed that  the set of directions into which the non-smoothness  is differentiated - in the "linearized" state equation - must be dense in a suitable (Bochner) space \cite[Remark 2.12]{st_coup}, \cite[Lem.\,5.2]{paper}.
The density of the set of directions into which $\max$ is differentiated, see \eqref{eq:ode_lin_q1}, is indeed available, as the first step of the proof of Theorem \ref{thm:ss_qsep} shows. This allowed us to improve the optimality system \eqref{eq:adjjj} from the previous section. However, the  non-differentiable function $f$ requires a similar density property too, which reads as follows
 \begin{equation}\label{hh}
 \{  \HH'(\bar q;S_1'(\ll;\dl)):\dl \in H^1(0,T;\lo)\} \dense \llz,
\end{equation}where $S_1$ denotes the first component of the control-to-state operator $S: \llU \ni \ell \mapsto (q,\varphi) \in   \hhyn \times \llu$.  By taking a look at the "linearized" state equation \eqref{eq:ode_lin_q1}, we see that \eqref{hh} is not to be expected, due to the lack of surjectivity of the mapping $\max'(\bar z;\cdot)$. Thus, the methods from \cite{st_coup, paper} are restricted to one non-smoothness and permit us to improve the limit optimality system \eqref{eq:adjjj} only up to a certain point. Thus, the  strong stationarity for  the control of \eqref{eq:optt} remains an open question.
\end{remark}

\begin{appendix}
\section{}\label{sec:a}

\begin{proofof}{Lemma \ref{lem:loc}}

The arguments  are well-known \cite{barbu}  and can be found  in \cite[App.\,B]{st_coup} for the case that $ \F(q)$ is constant and the control space is $\llz$ instead of $\hs$. 

(I) Let $\varepsilon>0$ be arbitrary, but fixed. 
We begin by recalling the smooth state equation appearing in \eqref{eq:p_eps}:
\begin{subequations}\label{eq:syst_diff_e}
\begin{gather}
\dot q(t)   =   \frac{1}{\epsilon} \maxeps (-\beta(q(t)-\varphi(t))-(f_\eps \circ \HH)(q)(t))\ \text{ in }\lo, \quad q(0) = 0,    \label{eq:syst_diff11_e} 
\\- \alpha \Delta \varphi(t) + \beta \,\varphi(t)   
= \beta q(t) +\ell(t) \quad  \text{in }H^{1}(\Omega)^\ast, \quad \text{a.e.\ in }(0,T).\label{eq:syst_diff12_e} 
\end{gather}
\end{subequations}
  By employing the exact same arguments as in the proof of Proposition \ref{lem:ode}, one infers that \eqref{eq:syst_diff_e} admits a unique solution $(q_\eps,\varphi_\eps) \in   \hhyn \times \llu $ for every  $\ell \in \llU$, which allows us to define the regularized solution mapping $$S_\varepsilon:\llU \ni \ell \mapsto (q_\eps,\varphi_\eps) \in   \hhyn \times \llu.$$ The operator $S_{\varepsilon}$ is G\^{a}teaux-differentiable and its derivative
 at $\ell  \in \llU$ in direction $\dl \in \llU$, i.e., $ (\dd,\df):=S'_{\varepsilon}(\ell)(\dl)$, is the unique solution of   \small
 \begin{equation}\label{eq:ode_lin_q_e}\begin{aligned}
\dot \dd(t)&=  \frac{1}{\epsilon} \maxeps ' (z_\eps(t))\big( -\beta(\dd(t)-\df(t))-(f_\eps \circ \HH)'(q_\eps)(\dd)(t)\big)\ \text{ in }\lo, \quad \dd(0)=0, \\
&- \alpha \Delta \df(t) + \beta \,\df(t) = \beta \dd(t) +\dl(t) \quad  \text{in }H^{1}(\Omega)^\ast, \quad \text{a.e.\ in }(0,T),
\end{aligned}
 \end{equation}\normalsize  where we abbreviate $ z_\eps:=-\beta(q_\eps-\varphi_\eps)-(f_\eps \circ \HH)(q_\eps)$. By arguing as in the proof of  Lemma \ref{lem:lip_S} we deduce that $S_\varepsilon: \llU  \to   \hhyn \times \llu$ is Lipschitz continuous (with constant independent of $\varepsilon$). Moreover, we have the convergence
 \small
  \begin{equation}\label{s_conv}
  S_\varepsilon(\ell_{\varepsilon}) \to S(\ell)   \quad \text{in } \hhyn \times \llu,
  \end{equation}\normalsize for $\ell_{\varepsilon} \to \ell \text{ in } \llU$.
  To see this, one first shows that $S_\varepsilon(\ell) \to S(\ell)$, which follows by estimating as in the proof of Lemma \ref{lem:lip_S} and by using \eqref{eq:f_h} applied for $f_\eps$ along with \eqref{conv_f_eps}. Then, \eqref{s_conv} is a consequence of the  Lipschitz continuity of $S_\varepsilon$ (with constant independent of $\eps$).

\vspace{-0.19cm}(II)  Next, we focus on proving that $\ll$ can be approximated via local minimizers of optimal control problems governed by \eqref{eq:syst_diff_e}. To this end, let $B_{H^1(0,T;\lo )}(\ll,  \rho)$ be the ball of local optimality of $\ll$ and consider the smooth (reduced) optimal control problem
\begin{equation}\tag{$P_{\eps}^\rho$}\label{eq:q_eps}
 \left.
 \begin{aligned}
  \min_{\ell \in H^1(0,T;\lo)} \quad & J(S_{\eps}(\ell),\ell)+\frac{1}{2}\|\ell-\ll\|_{H^1(0,T;\lo)}^2
  \\
 \text{s.t.} \quad & \ell \in B_{H^1(0,T;\lo)}(\ll,  \rho).
 \end{aligned}
 \quad \right\}
\end{equation}  
By arguing as in the proof of Proposition \ref{prop:ex},
we see that \eqref{eq:q_eps} admits a global solution $\ell_{\eps} \in H^1(0,T;\lo)$.  Since $\ell_{\eps}  \in B_{H^1(0,T;\lo)}(\ll,  \rho),$ we can select a subsequence with 
\begin{equation}\label{l}
\ell_{\eps}  \weakly \widetilde \ell \quad \text{in }H^1(0,T;\lo),
\end{equation}
where $\widetilde \ell  \in B_{H^1(0,T;\lo)}(\ll,  \rho).$
For simplicity, we abbreviate in the following \small  \begin{subequations} \begin{gather}
\JJ(\ell):=J(S(\ell),\ell) ,\label{jj}\\
\JJ_{\eps}(\ell):=J(S_{\eps}(\ell),\ell)+\frac{1}{2}\|\ell-\ll\|_{H^1(0,T;\lo)}^2  \label{jn}
\end{gather}\end{subequations} \normalsize for all $\ell \in H^1(0,T;\lo)$.
Due to \eqref{s_conv} and Assumption \ref{assu:j_ex},  it holds 
  \begin{equation}\label{j}
  \JJ(\ll)\overset{\eqref{jj}}{=}J(S(\ll),\ll)=\lim_{\eps \to 0} J(S_{\eps}(\ll),\ll)\overset{\eqref{jn}}{=}\lim_{\eps \to 0}  \JJ_{\eps}(\ll) \geq \limsup_{\eps \to 0}   \JJ_{\eps}(\ell_{\eps}),
  \end{equation}\normalsize where for the last inequality we relied on the fact that $\ell_{\eps}$ is a global minimizer of \eqref{eq:q_eps} and that $\ll$ is admissible for \eqref{eq:q_eps}. In view of  \eqref{jn}, \eqref{j} can be continued as 
   \begin{equation}\label{j1}\begin{aligned}
  \JJ(\ll) &\geq \limsup_{\eps \to 0}   J(S_{\eps}(\ell_{\eps}),\ell_{\eps})+\frac{1}{2}\|\ell_{\eps}-\ll\|_{H^1(0,T;\lo)}^2
   \\&\quad   \geq \liminf_{\eps \to 0}   J(S_{\eps}(\ell_{\eps}),\ell_{\eps})+\frac{1}{2}\|\ell_{\eps}-\ll\|_{H^1(0,T;\lo)}^2
      \\&\qquad   \geq    J(S(\widetilde \ell), \widetilde \ell)    +\frac{1}{2}\|\widetilde \ell-\ll\|_{H^1(0,T;\lo)}^2
  \geq   \JJ(\ll),
\end{aligned}  \end{equation}where we used again \eqref{s_conv} in combination with the compact embedding $\hs \embed \embed \llU$, and the continuity of $j$, see Assumption \ref{assu:j_ex}; note that for the last inequality in \eqref{j1} we employed the fact that $\widetilde \ell \in B_{H^1(0,T;\lo )}(\ll,  \rho)$.
From \eqref{j1} we obtain that $\widetilde \ell=\ll$ and 
$$\JJ(\ll) = \lim_{\eps \to 0}   J(S_{\eps}(\ell_{\eps}),\ell_{\eps})+\frac{1}{2}\|\ell_{\eps}-\ll\|_{H^1(0,T;\lo)}^2=J(S(\widetilde \ell), \widetilde \ell)    +\frac{1}{2}\|\widetilde \ell-\ll\|_{H^1(0,T;\lo)}^2.$$Since $J(S_{\eps}(\ell_{\eps}),\ell_{\eps}) \to J(S(\widetilde \ell), \widetilde \ell)    $,
one   has  the convergence  
 \begin{equation}\label{l_conv1}
\ell_{\eps} \to \ll \quad  \text{in }H^1(0,T;\lo), \end{equation}
where we also relied on \eqref{l}.
As a consequence, \eqref{s_conv} yields
\begin{equation}\label{y_conv1}
  S_\varepsilon(\ell_{\varepsilon}) \to S(\ll)  \quad \text{in } \hhyn \times \llu.
  \end{equation}
A classical argument  finally shows that $\ell_\eps$ is a local minimizer of $\min_{\ell \in \hs} \JJ_\eps(\ell)$ for $\eps>0$ sufficiently small.
\end{proofof}
\end{appendix}

\section*{Acknowledgment}This work was supported by the DFG grant BE 7178/3-1 for the project "Optimal Control of Viscous
Fatigue Damage Models for Brittle Materials: Optimality Systems".

\bibliographystyle{plain}
\bibliography{strong_stat}

\begin{thebibliography}{10}

\bibitem{alessi}
R.~Alessi, V.~Crismale, and G.~Orlando.
\newblock Fatigue effects in elastic materials with variational damage models:
  A vanishing viscosity approach.
\newblock {\em Journal of Nonlinear Science}, 29:1041--1094, 2019.

\bibitem{alessi1}
R.~Alessi, S.~Vidoli, and L.~De~Lorenzis.
\newblock Variational approach to fatigue phenomena with a phase-field model:
  the one-dimensional case.
\newblock {\em Engineering Fracture Mechanics}, 190:53--73, 2018.

\bibitem{Barbu:1981:NCD}
V.~Barbu.
\newblock Necessary conditions for distributed control problems governed by
  parabolic variational inequalities.
\newblock {\em SIAM Journal on Control and Optimization}, 19(1):64--86, 1981.

\bibitem{barbu}
V.~Barbu.
\newblock {\em Optimal control of variational inequalities}.
\newblock Research notes in mathematics 100, Pitman, Boston-London-Melbourne,
  1984.

\bibitem{st_coup}
L.~Betz.
\newblock Strong stationarity for optimal control of a non-smooth coupled
  system: Application to a viscous evolutionary {VI} coupled with an elliptic
  {PDE}.
\newblock {\em SIAM J.\ on Optimization}, 29(4):3069--3099, 2019.

\bibitem{mcrf}
L.~Betz.
\newblock Strong stationarity for a highly nonsmooth optimization problem with
  control constraints.
\newblock {\em Mathematical Control and Related Fields,
  doi={10.3934/mcrf.2022047}}, 2022.

\bibitem{cc}
C.~Christof.
\newblock Sensitivity analysis and optimal control of obstacle-type evolution
  variational inequalities.
\newblock {\em SIAM J. Control Optim.}, 57(1):192--218, 2019.

\bibitem{brok_ch}
C.~Christof and M.~Brokate.
\newblock Strong stationarity conditions for optimal control problems governed
  by a rate-independent evolution variational inequality.
\newblock arXiv:2205.01196, 2022.

\bibitem{cr_meyer}
C.~Christof, C.~Clason, C.~Meyer, and S.~Walther.
\newblock Optimal control of a non-smooth, semilinear elliptic equation.
\newblock {\em Mathematical Control and Related Fields}, 8(1):247--276, 2018.

\bibitem{quasi_nonsmooth}
C.~Clason, V.H. Nhu, and A.~R\"osch.
\newblock Optimal control of a non-smooth quasilinear elliptic equation.
\newblock {\em Mathematical Control and Related Fields}, 11(3):521--554, 2021.

\bibitem{a1}
V.~Crismale, G.~Lazzaroni, and G.~Orlando.
\newblock Cohesive fracture with irreversibility: quasistatic evolution for a
  model subject to fatigue.
\newblock {\em Math. Models Methods Appl. Sci.}, 28:1371--1412, 2018.

\bibitem{DelosReyes-Meyer}
J.~C. De~los Reyes and C.~Meyer.
\newblock Strong stationarity conditions for a class of optimization problems
  governed by variational inequalities of the second kind.
\newblock {\em Journal of Optimization Theory and Applications},
  168(2):375--409, 2016.

\bibitem{hackl}
B.J. Dimitrijevic and K.~Hackl.
\newblock A method for gradient enhancement of continuum damage models.
\newblock {\em Technische Mechanik}, 28(1):43–52, 2008.

\bibitem{emm}
E.~Emmrich.
\newblock {\em Gew{\"o}hnliche und Operator Differentialgleichungen}.
\newblock Vieweg, Wiesbaden, 2004.

\bibitem{FK06}
M.~Fr\'emond and N.~Kenmochi.
\newblock Damage problems for viscous locking materials.
\newblock {\em Adv. Math. Sci. Appl.}, 16(2):697–716, 2006.

\bibitem{fn96}
M.~Fr\'emond and B.~Nedjar.
\newblock Damage, gradient of damage and principle of virtual power.
\newblock {\em Int. J. Solids Struct.}, 33(8):1083–1103, 1996.

\bibitem{Friedman1987}
A.~Friedman.
\newblock Optimal control for parabolic variational inequalities.
\newblock {\em SIAM Journal on Control and Optimization}, 25(2):482--497, 1987.

\bibitem{oc2}
K.~Hammerum, P.~Brath, and N.~K. Poulsen.
\newblock A fatigue approach to wind turbine control.
\newblock {\em J. of Physics: Conf. Series}, 75:012--081, 2007.

\bibitem{He1987}
Z.X. He.
\newblock State constrained control problems governed by variational
  inequalities.
\newblock {\em SIAM Journal on Control and Optimization}, 25:1119--1144, 1987.

\bibitem{HintermuellerKopacka2008:1}
M.~Hinterm{\"u}ller and I.~Kopacka.
\newblock Mathematical programs with complementarity constraints in function
  space: {C}- and strong stationarity and a path-following algorithm.
\newblock {\em SIAM Journal on Optimization}, 20(2):868--902, 2009.

\bibitem{kaballo}
W.~Kaballo.
\newblock {\em Aufbaukurs Funktionalanalysis und Operatortheorie}.
\newblock Springer Verlag, Heidelberg, 2014.

\bibitem{knees}
D.~Knees, R.~Rossi, and C.~Zanini.
\newblock A vanishing viscosity approach to a rate-independent damage model.
\newblock {\em Mathematical Models and Methods in Applied Sciences},
  23(4):565--616, 2013.

\bibitem{paper}
C.~Meyer and L.M. Susu.
\newblock Optimal control of nonsmooth, semilinear parabolic equations.
\newblock {\em SIAM Journal on Control and Optimization}, 55(4):2206--2234,
  2017.

\bibitem{paper1}
C.~Meyer and L.M. Susu.
\newblock Analysis of a viscous two-field gradient damage model. {P}art {I}:
  Existence and uniqueness.
\newblock {\em Z.\ Anal.\ Anwend.}, 38(3):249--286, 2019.

\bibitem{paper2}
C.~Meyer and L.M. Susu.
\newblock Analysis of a viscous two-field gradient damage model. {P}art {II}:
  Penalization limit.
\newblock {\em Z.\ Anal.\ Anwend.}, 38(4), 2019.

\bibitem{mp76}
F~Mignot.
\newblock Contr{\^{o}}le dans les in{\'{e}}quations variationelles elliptiques.
\newblock {\em Journal of Functional Analysis}, 22(2):130--185, 1976.

\bibitem{mp84}
F.~Mignot and J.-P. Puel.
\newblock Optimal control in some variational inequalities.
\newblock {\em SIAM Journal on Control and Optimization}, 22(3):466--476, 1984.

\bibitem{oc1}
I.~Munteanu, A.~I. Bratcu, N.-A. Cutululis, and E.~Ceanga.
\newblock {\em Optimal control of wind energy systems: towards a global
  approach}.
\newblock Springer Science $\&$ Business Media, 2008.

\bibitem{rl}
Ritchie R.O. and Launey M.E.
\newblock Crack growth in brittle and ductile solids.
\newblock In {\em Wang Q.J., Chung YW. (eds) Encyclopedia of Tribology}.
  Springer, Boston, 2013.

\bibitem{rockaf}
R.T. Rockafellar.
\newblock {\em Convex Analysis}.
\newblock Princeton University Press, Princeton, 1970.

\bibitem{ScheelScholtes2000}
H.~Scheel and S.~Scholtes.
\newblock Mathematical programs with complementarity constraints: Stationarity,
  optimality, and sensitivity.
\newblock {\em Mathematics of Operations Research}, 25(1):1--22, 2000.

\bibitem{schirotzek}
W.~Schirotzek.
\newblock {\em Nonsmooth Analysis}.
\newblock Springer, Berlin, 2007.

\bibitem{shapiro}
A.~Shapiro.
\newblock On concepts of directional differentiability.
\newblock {\em Journal of Optimization Theory and Applications},
  66(3):477--487, 1990.

\bibitem{sofonea}
M.~Sofonea and M.~Andaluzia.
\newblock {\em Variational Inequalities with Applications}.
\newblock Springer, New York, 2009.

\bibitem{stephens}
R.~Stephens, A.~Fatemi, R.~Stephens, and H.~Fuchs.
\newblock {\em Metal Fatigue in Engineering}.
\newblock A Wiley-Interscience publication, Wiley, New York, 2000.

\bibitem{susu_dam}
L.M. Susu.
\newblock Optimal control of a viscous two-field gradient damage model.
\newblock {\em GAMM-Mitt.}, 40(4):287 -- 311, 2018.

\bibitem{tiba}
D.~Tiba.
\newblock {\em Optimal control of nonsmooth distributed parameter systems}.
\newblock Springer, 1990.

\bibitem{troe}
F.~Tr{\"o}ltzsch.
\newblock {\em Optimal Control of Partial Differential Equations}, volume 112
  of {\em Graduate Studies in Mathematics}.
\newblock American Mathematical Society, Providence, 2010.
\newblock Theory, methods and applications, Translated from the 2005 German
  original by J{\"u}rgen Sprekels.

\bibitem{e_qvi}
G.~Wachsmut.
\newblock Elliptic quasi-variational inequalities under a smallness assumption:
  uniqueness, differential stability and optimal control.
\newblock {\em Calc. Var.}, 82(59), 2020.

\bibitem{wachsm_2014}
G.~Wachsmuth.
\newblock Strong stationarity for optimal control of the obstacle problem with
  control constraints.
\newblock {\em SIAM Journal on Optimization}, 24(4):1914--1932, 2014.

\end{thebibliography}

\end{document}